\documentclass[12pt,oneside]{amsart}
\usepackage{amssymb,amscd,amsthm,verbatim}
\usepackage[left=1in,right=1in,top=1in,bottom=1in]{geometry}
\usepackage{latexsym}
\usepackage{color}
\usepackage{xcolor}
\usepackage{graphicx}
\usepackage{mathrsfs}
\usepackage[all,cmtip]{xy}
\usepackage{bbm}
\usepackage{stmaryrd}
\usepackage{xr}
\usepackage{cite}

\usepackage[
	colorlinks,
	linkcolor={black},
	citecolor={black},
	urlcolor={black}
]{hyperref}

\date{\today}

\newcommand{\bbC}{{\mathbb{C}}}
\newcommand{\bbD}{{\mathbb{D}}}
\newcommand{\bbN}{{\mathbb{N}}}
\newcommand{\bbP}{{\mathbb{P}}}
\newcommand{\bbR}{{\mathbb{R}}}
\newcommand{\bbT}{{\mathbb{T}}}
\newcommand{\bbX}{{\mathbb{X}}}
\newcommand{\bbZ}{{\mathbb{Z}}}

\newcommand{\calA}{{\mathcal{A}}}
\newcommand{\calE}{{\mathcal{E}}}
\newcommand{\calF}{{\mathcal{F}}}
\newcommand{\calH}{{\mathcal{H}}}
\newcommand{\calL}{{\mathcal{L}}}
\newcommand{\calM}{{\mathcal{M}}}
\newcommand{\calU}{{\mathcal{U}}}
\newcommand{\calZ}{{\mathcal{Z}}}

\newcommand{\frA}{{\mathfrak{A}}}
\newcommand{\frF}{{\mathfrak{F}}}

\newcommand{\sfa}{{\mathsf{a}}}
\newcommand{\sfb}{{\mathsf{b}}}

\newcommand{\tr}{{\mathrm{Tr}}}
\newcommand{\UH}{{\calU\calH}}
\newcommand{\unstSect}{\mathsf{E}^+}
\newcommand{\stabSect}{\mathsf{E}^-}
\newcommand{\stunSect}{\mathsf{E}^\pm}

\newcommand{\GL}{{\mathrm{GL}}}
\newcommand{\SL}{{\mathrm{SL}}}

\renewcommand{\Re}{{\mathrm{Re}}}
\renewcommand{\Im}{{\mathrm{Im}}}
\newcommand{\rot}{{\mathrm{rot}}}

\newtheorem{theorem}{Theorem}[section]

\newtheorem{prop}[theorem]{Proposition}

\theoremstyle{definition}

\newtheorem{definition}[theorem]{Definition}
\newtheorem{example}[theorem]{Example}
\newtheorem{remark}[theorem]{Remark}

\theoremstyle{plain}

\allowdisplaybreaks 

\numberwithin{equation}{section}

\newcommand{\set}[1]{\left\{#1\right\}}

\DeclareMathOperator{\supp}{supp}

\begin{document}

\title[Gap Labelling for Lee--Yang Zeros]{Gap Labels for Zeros of the Partition Function of the 1D Ising Model via the Schwartzman Homomorphism}

\author[D.\ Damanik]{David Damanik}
\address{Department of Mathematics, Rice University, Houston, TX~77005, USA}
\email{damanik@rice.edu}

\author[M.\ Embree]{Mark Embree}
\address{Department of Mathematics, Virginia Tech, Blacksburg, VA 24061, USA}
\email{embree@vt.edu}

\author[J.\ Fillman]{Jake Fillman}
\address{Department of Mathematics, Texas State University, San Marcos, TX 78666, USA}
\email{fillman@txstate.edu}

\dedicatory{Dedicated to the memory of Uwe Grimm}

\maketitle

\begin{abstract}
Inspired by the 1995 paper of Baake--Grimm--Pisani, we aim to explain the empirical observation that the distribution of Lee--Yang zeros corresponding to a one-dimensional Ising model appears to follow the gap labelling theorem. This follows by combining two main ingredients: first, the relation between the transfer matrix formalism for 1D Ising model and an ostensibly unrelated matrix formalism generating the Szeg\H{o} recursion for orthogonal polynomials on the unit circle, and second, the gap labelling theorem for CMV matrices.
\end{abstract}

\setcounter{tocdepth}{1}\tableofcontents

\hypersetup{
	linkcolor={black!30!blue},
	citecolor={red},
	urlcolor={black!30!blue}
}

\section{Introduction}

\subsection{Inspiration}

Since their discovery, gap labelling theorems have been a useful tool in the analysis of operators. In an abstract formulation, a gap labelling theorem says that if an operator family is generated by an ergodic process by continuously sampling along orbits of the process, then there is a countable subgroup of $\bbR$ that describes the distribution of eigenvalues of the operators. Moreover, this group depends on the ergodic process, but does not depend on the continuous function by which the operators are generated. The $K$-theory formulation of the gap labelling theorem, due to Bellissard and coworkers, realizes this subgroup as the range of a normalized trace on a suitable $C^*$ algebra \cite{Bel1992b, Bel2003}. This approach is further elucidated for substitution models in \cite{BaaGriJos1993, BelBovGhe1992}. The Johnson formulation realizes the gap labels in terms of the range of the Schwartzman asymptotic cycle for one-dimensional differential and finite difference operators \cite{Johnson1986JDE, JohnMos1982CMP}. For additional details and applications, see also \cite{DFGap} and references therein.

In some situations, distributions of zeros or eigenvalues appear to obey a law similar to the one predicted by a gap labelling theorem, even if no operators seem to be present. In 1995, Baake--Grimm--Pisani observed that the one-dimensional ferromagnetic Ising model appears to be just such a model when the magnetic couplings are chosen according to the Fibonacci substitution sequence \cite{BaaGriPis1995JSP}. The purpose of this note is to explain how certain unitary operators (hence gap labelling theorems) enter the picture and to contextualize the observations of \cite{BaaGriPis1995JSP}.

In short, the zeros of a Lee--Yang partition function can be identified with zeros of the trace of a $2 \times 2$ matrix propagator for the Szeg\H{o} recursion for orthogonal polynomials generated by a measure on the unit circle \cite{DamMunYes2013JSP, YessenThesis}. In turn, these zeros can be shown to be eigenvalues of unitary operators derived from such orthogonal polynomials. On the other hand, there exists a gap labelling theory for such unitary operators, which is a result of Geronimo--Johnson \cite{GeronimoJohnson1996JDE}. The rest of the note will spell this out in more detail, and we conclude with a gallery of examples.

\subsection{Ferromagnetic Ising Models on the Line}

Let us begin by defining objects associated with an Ising model on the line. This is not meant to be an exhaustive overview; we just collect the objects and results that we need to exhibit our main points. For additional background and history, we direct the reader to \cite{Baxter1989, Brush1967RMP} and references therein.

To specify a one-dimensional ferromagnetic Ising model, choose a sequence of magnetic couplings $\{J_n\}_{n=1}^\infty$ with $J_n>0$ for all $n$. 

For each $N \in \bbN$, denote $\Lambda_N = \{\pm 1\}^N = \{\pm\}^N$. Both versions of $\Lambda_N$ are convenient in certain formulas, so we freely pass between the two representations. On the lattice $\{1,2,\ldots,N\}$, the nearest neighbor Ising model with constant field $H$ is specified by the energy functional
\begin{equation}
E(\sigma) := -\frac{1}{k_B \tau} \sum_{n=1}^N \left( J_n \sigma_n\sigma_{n+1} + H\sigma_n\right), \quad \sigma= (\sigma_1,\ldots,\sigma_N) \in \Lambda_N,
\end{equation}
where $H$ denotes the magnetic field, $\tau>0$ is the temperature, $k_B>0$ is the Boltzmann constant, and $\sigma$ satisfies the periodic boundary condition 
\begin{equation} \label{eq:sigmaperbc}
\sigma_{N+1} = \sigma_1.
\end{equation} 
For convenience of notation, we introduce $p_n=J_n/(k_B \tau)$ and $q = H/(k_B\tau)$ so that
\begin{equation}
E(\sigma) = E(\sigma,q) := - \sum_{n=1}^N \left( p_n \sigma_n\sigma_{n+1} + q\sigma_n\right), \quad \sigma \in \Lambda_N.
\end{equation}
In physical applications, one is often interested in the \emph{Gibbs state} in which $\mathbb{P}(\sigma)$, the probability of the configuration $\sigma$, is proportional to $\exp(-E(\sigma))$, since this is the probability distribution on $\Lambda_N$ that maximizes the entropy $-\sum_\sigma \mathbb{P}(\sigma) \log \mathbb{P}(\sigma)$.  Naturally, the corresponding normalization constant, known as the \emph{partition function}, plays an important role. More precisely, the partition function is defined by 
\begin{equation}
Z_N(q):= \sum_{\sigma \in \Lambda_N} e^{-E(\sigma,q)}, \quad N \in \bbN.
\end{equation}
Introduce the variables
\begin{equation} \label{eq:wbetadef}
\zeta = e^{q}, \quad \beta_n = e^{p_n} 
\end{equation}
so that $Z_N$ can be viewed as a function of the variable $\zeta$:
\begin{equation} \label{eq:calZNdef}
\calZ_N(\zeta) = \sum_{\sigma \in \Lambda_N} \prod_{n=1}^N \beta_n^{\sigma_n\sigma_{n+1}} \zeta^{\sigma_n}.
\end{equation}
Due to the Lee--Yang theorem \cite{LeeYang1952PR}, zeros of $\calZ_N$ lie on the unit circle
\[ \partial \bbD := \{\zeta \in \bbC : |\zeta|=1\}.\]
Later on, we will see that the zeros of $\calZ_N$ are the eigenvalues of a suitable unitary operator, which gives another way to see that they lie on $\partial \bbD$ (compare Propositions~\ref{prop:DNzerosAreEigs} and \ref{prop:IsingzerosToDN}).

\subsection{The Ergodic Setting}

The examples that we will study in the present work are generated by sampling along orbits of ergodic topological dynamical systems. Let us make this more precise.

 Suppose $(\Omega,T,\mu)$ is an ergodic topological dynamical system (we will review definitions and results from ergodic theory in Section~\ref{sec:background}), denote $\bbR_+ = \{x \in \bbR : x>0\}$, and consider $g \in C(\Omega,\bbR_+)$. For each $\omega \in \Omega$, one obtains a realization of a ferromagnetic Ising model by taking $p_n = p_n(\omega) = g(T^n\omega)$. We denote the dependence on $\omega$ by writing, for example,
\begin{equation}
\calZ_N(\zeta,\omega) = \sum_{\sigma \in \Lambda_N} \prod_{n=1}^N e^{g(T^n\omega)\sigma_n\sigma_{n+1}} \zeta^{\sigma_n}
\end{equation}
for the partition function. Let us say that $\zeta \in \partial \bbD$ is in a \emph{spectral gap} of the Ising model if for some $\varepsilon>0$, a.e.\ $\omega$, and sufficiently large $N$ there are no zeros of $\calZ_N(\zeta,\omega)$ in an $\varepsilon$-neighborhood of $\zeta$.

\begin{theorem} \label{t:isingGL}
Let $(\Omega,T,\mu)$ denote an ergodic topological dynamical system such that $\Omega = \supp \mu$. There is a countable group $\frA = \frA(\Omega,T,\mu)$ such that the following statement holds true. 

Given $g \in C(\Omega,\bbR_+)$, let $\calZ_N(\zeta,\omega)$ denote the associated partition functions. If $\zeta_1,\zeta_2 \in \partial \bbD$ both lie in a spectral gap of the associated Ising model, then
\begin{equation} \label{eq:isingGL}
\lim_{N\to\infty} \frac{1}{N} \#\{\zeta \in [\zeta_1,\zeta_2] : \calZ_N(\zeta,\omega)=0\} \in \frA(\Omega,T,\mu), \quad \text{a.e.\ } \omega \in \Omega.
\end{equation}
where $[\zeta_1, \zeta_2]$ denotes the closed arc from $\zeta_1$ to $\zeta_2$ in the counterclockwise direction.
\end{theorem}

\begin{remark}
Let us make some comments.
\begin{enumerate}
\item[(a)] The group $\frA(\Omega,T,\mu)$ may be computed explicitly in many cases of interest. Since it arises from the application of a homomorphism studied by Schwartzman \cite{Schwartzman1957Annals}, it is sometimes called the \emph{Schwartzman group} of $(\Omega,T,\mu)$. We will describe $\frA(\Omega,T,\mu)$ more precisely in Section~\ref{sec:background}. In Section~\ref{sec:gallery}, we will look at some specific examples in which $\frA$ can be computed.
\item[(b)] In the case in which $(\Omega,T,\mu)$ is the strictly ergodic subshift generated by the Fibonacci substitution, Baake--Grimm--Pisani observed the conclusion of Theorem~\ref{t:isingGL} empirically \cite{BaaGriPis1995JSP}. We will discuss this further in Section~\ref{sec:gallery}.
\item[(c)] Theorem~\ref{t:isingGL} follows by combining some theorems and observations from a few different papers that came about since the publication of \cite{BaaGriPis1995JSP}. \end{enumerate}
\end{remark}

Uwe Grimm was an exceptionally generous and encouraging colleague who enjoyed finding surprising connections between ostensibly different mathematical problems. We hope that Uwe would have appreciated how recent developments in mathematical physics shed new light on his earlier observations.

The rest of the paper is laid out as follows. In Section~\ref{sec:background}, we discuss some background about dynamical systems and CMV matrices. Section~\ref{sec:LY2Sz} explains how the partition function may be related to a polynomial derived from a suitable CMV matrix, and then Section~\ref{sec:Sz2Schwartman} explains how to prove Theorem~\ref{t:isingGL}. We conclude with a discussion of specific classes of examples in Section~\ref{sec:gallery} as well as some relevant plots.

\subsection*{Acknowledgements} We are grateful to Michael Baake for helpful comments. D.D.\ was supported in part by NSF grants DMS--1700131 and DMS--2054752. J.F.\ was supported in part by NSF grant DMS--2213196 and Simons Foundation Collaboration Grant \#711663.  The authors thank the American Institute of Mathematics for hospitality and support through the SQuaRE program during a remote meeting in January~2021 and a January~2022 visit, during which this work was initiated. The authors also gratefully acknowledge support from the Simons Center for Geometry and Physics, at which some of this work was done.

\section{Background} \label{sec:background}

Let us begin by reviewing some relevant background. Since Theorem~\ref{t:isingGL} is proved by connecting ideas from dynamical systems and CMV matrices to the Ising model, we introduce the relevant notions from topological dynamics, the general theory of CMV matrices, and the theory of CMV matrices with dynamically defined coefficients.

\subsection{Odds and Ends from Dynamical Systems} \label{ssec:topdyn}

Let us briefly review some terminology and relevant results from dynamical systems. For further reading, one may consult textbook treatments such as Brin--Stuck \cite{BrinStuck2015Book}, Katok--Hasselblatt \cite{KatokHassel1995Book}, and Walters \cite{Walters1982:ErgTh}.

\begin{definition}
By a \emph{topological dynamical system}, we mean an ordered pair $(\Omega,T)$ in which $\Omega$ is a compact metric space and $T:\Omega \to \Omega$ is a homeomorphism. A Borel probability measure $\mu$ on $\Omega$ is called $T$\emph{-invariant} if $\mu(T^{-1}B) = \mu(B)$ for each Borel set $B \subseteq \Omega$. A $T$-invariant Borel probability measure $\mu$ is called \emph{ergodic} (with respect to $T$) if $\mu(E) \in \{0,1\}$ whenever $T^{-1}E=E$. In this case, we say that the triple $(\Omega,T,\mu)$ is an ergodic topological dynamical system.
\end{definition}

\begin{definition}Suppose $(\Omega,T)$ denotes a topological dynamical system. Given a continuous map $A:\Omega \to \GL(2,\bbC)$, the associated \emph{cocycle} is the skew product
\begin{equation}
(T,A):\Omega \times \bbC^2 \to \Omega \times \bbC^2, \quad (\omega,v)\mapsto (T\omega,A(\omega)v).
\end{equation}
The iterates of $A$ are then defined by $(T,A)^n = (T^n,A^n)$, which the reader can check implies
\begin{equation} \label{eq:cocycleIteratesDef}
A^n(\omega) =
\begin{cases}
A(T^{n-1}\omega) \cdots A(T\omega)A(\omega), & n \geq 1;\\
I, & n =0;\\ 
[A^{-n}(T^n\omega)]^{-1}, & n \leq -1.
\end{cases}
\end{equation}
\end{definition}

\begin{definition}
Consider a continuous cocycle $(T,A)$ over a topological dynamical system $(\Omega, T)$.
\begin{enumerate}
\item[(a)] We say that $(T,A)$ is \emph{uniformly hyperbolic} if for constants $c,\lambda > 0$ one has
\begin{equation}
\|A^n(\omega)\| \geq ce^{\lambda |n|}, \quad \forall \omega \in \Omega, \ n \in \bbZ.
\end{equation}

\item[(b)] We say that $(T,A)$ enjoys an \emph{exponential dichotomy} if there exist continuous maps $\stabSect,\unstSect:\Omega \to \bbC\bbP^1$ such that
\begin{equation}
A(\omega) \stunSect(\omega) = \stunSect(T\omega),
\end{equation}
and constants $C,\lambda >0$ such that
\begin{equation}
\|A^n(\omega)v_+\| , \ \|A^{-n}(\omega)v_-\| \leq Ce^{-\lambda n}, \quad \forall n \in \bbN, \ \omega \in \Omega,
\end{equation}
for all unit vectors $v_\pm \in \stunSect(\omega)$.
\end{enumerate}
\end{definition}

If $|\det A(\omega)| =1$, then (a) and (b) are equivalent. See \cite{DFLY2016DCDS} for proofs.

\subsection*{The Schwartzman Homomorphism}
As before, let $(\Omega,T,\mu)$ denote an ergodic topological dynamical system. To define the Schwartzman homomorphism and the associated groups, one needs a flow, that is, a continuous-time dynamical system. The most natural way to produce a continuous-time dynamical system that interpolates a discrete-time system such as $(\Omega,T)$ is to form the \emph{suspension}. To be more specific, the \emph{suspension} of $(\Omega,T,\mu)$, denoted $(X,\tau,\nu)$, is defined as follows. The space $X$ is given by
\begin{equation}
X = \Omega \times \bbR/ \! \sim, \quad \text{ where } (\omega,t) \sim (\omega',t') \iff t-t' \in \bbZ \text{ and } T^{t-t'}\omega = \omega'. \end{equation} 
We write $[\omega,t]$ for the class of $(\omega,t)$ in $X$. The flow on $X$, denoted by $\tau$, is the natural projection of the translation action of $\bbR$, that is, 
\begin{equation}
\tau^s([\omega,t]) = [\omega,t+s], \quad [\omega,t] \in X.
\end{equation}
 Finally, $\nu$ is the natural measure on $X$ given by
\begin{equation}
\int_X f\, d\nu = \int_\Omega \int_0^1 f([\omega,t]) \, d\mu(\omega) \, dt.
\end{equation}

Recall that $\phi_0,\phi_1 \in C(X,\bbT)$ are called \emph{homotopic}, denoted $\phi_0 \sim \phi_1$, if there exists a continuous $F:X \times [0,1] \to \bbT$ such that $F(\cdot,j) = \phi_j$ for $j=0,1$. Let $C^\sharp(X,\bbT) = C(X,\bbT)/\!\sim$ denote the set of homotopy classes of continuous maps $X \to \bbT$. Given $\phi \in C(X,\bbT)$, $x \in X$, one can lift the map $\phi_x:t \mapsto \phi(\tau^t x)$ to $\psi_x : \bbR \to \bbR$. From  \cite{Schwartzman1957Annals}, there exists a real number $\rot(\phi) \in \bbR$ such that
\[\rot(\phi) = \lim_{t \to \infty} \frac{\psi_x(t)}{t}, \quad \nu\text{-a.e.\ } x \in X.\]
The induced map $\frF_\nu:C^\sharp(X,\bbT) \to \bbR$ given by
\[ \frF_\nu([\phi]) = \rot(\phi) \quad \nu\text{-a.e.\ } x \in X, \]
is called the \emph{Schwartzman homomorphism}. When working with linear cocycles over a dynamical system, it is often convenient to work with maps into the projective line $\bbR\bbP^1$ instead of $\bbT$. For such maps, one can define $\frF_\nu$ by identifying $\bbR\bbP^1$ with $\bbT$ via the map $\bbT \ni \theta \mapsto \mathrm{span}\{(\cos\pi\theta,\sin\pi\theta)^\top\} \in \bbR\bbP^1$. Using this identification, if $\Lambda \in C(X,\bbR\bbP^1)$, one has
\begin{equation} \label{eq:schwartzmanFromDeltaArg}
\frF_\nu([\Lambda]) 
= \lim_{T\to \infty} \frac{1}{\pi T} \Delta_{\rm arg}^{[0,T]} \Lambda(\tau^t x), \quad \nu\text{-a.e.\ } x \in X,
\end{equation}
where $\Delta_{\rm arg}^I$ denotes the net change in the argument on the interval $I$. Since we have chosen to define the Schwartzman homomorphism and group by considering maps into $\bbT =   \bbR / \bbZ$ instead of $\bbR\bbP^1$, notice the factor of $\pi$ that appears in \eqref{eq:schwartzmanFromDeltaArg}.

\begin{definition}
With notation as above, the \emph{Schwartzman group} associated with $(\Omega,T,\mu)$, denoted $\frA(\Omega,T,\mu)$, is the range of the Schwartzman homomorphism, that is,
\begin{equation}
\frA(\Omega,T,\mu) = \frF_\nu(C^\sharp(X,\bbT)).
\end{equation}
\end{definition}

It is known and not hard to check that $\frA(\Omega,T,\mu)$ is a countable subgroup of $\bbR$ that contains $\bbZ$. Indeed, one can check that $C^\sharp(X,\bbT)$ has at most countably many elements and the (class of) the map $[\omega,t]\mapsto t \ \mathrm{mod} \ \bbZ$ is mapped to $1$ by $\frF_\nu$. The reader may see \cite{DFGap, ESO1} for details and further discussion. In Section~\ref{sec:gallery}, we will discuss some specific examples and identify their Schwartzman groups (without proofs, which can also be found in \cite{DFGap}).

\subsection{CMV Matrices}
Let us briefly review some aspects of the general theory of CMV matrices and Floquet theory for periodic CMV matrices. We refer the reader to the monographs \cite{Simon2005:OPUC1, Simon2005:OPUC2} for additional details and proofs.

Given a sequence $\{\alpha_n\}_{n \in \bbZ}$ with $\alpha_n \in \bbD$ for every $n$, the associated \emph{extended CMV matrix} $\calE = \calE_\alpha$ is given by
\begin{equation} \label{def:extcmv}
\calE
=
\begin{bmatrix}
\ddots & \ddots & \ddots & \ddots &&&&  \\
& \overline{\alpha_0}\rho_{-1} & \boxed{-\overline{\alpha_0}\alpha_{-1}} & \overline{\alpha_1}\rho_0 & \rho_1\rho_0 &&&  \\
& {\rho_0\rho_{-1}} & -{\rho_0}\alpha_{-1} & {-\overline{\alpha_1}\alpha_0} & -\rho_1 \alpha_0 &&&  \\
&&  & \overline{\alpha_2}\rho_1 & -\overline{\alpha_2}\alpha_1 & \overline{\alpha_3} \rho_2 & \rho_3\rho_2 & \\
&& & {\rho_2\rho_1} & -{\rho_2}\alpha_1 & -\overline{\alpha_3}\alpha_2 & -\rho_3\alpha_2 &    \\
&& && \ddots & \ddots & \ddots & \ddots &
\end{bmatrix},
\end{equation}
where $\rho_n = \sqrt{1-|\alpha_n|^2}$ and we use a box to denote the matrix element $\langle \delta_0,\calE_\alpha,\delta_0\rangle$. It is well known that $\calE$ enjoys a factorization $\calE = \calL\calM$, where $\calL$ and $\calM$ are block diagonal with $2 \times 2$ blocks. Namely,
\begin{align}
\calL & = \bigoplus \Theta(\alpha_{2n}) \\
\calM & = \bigoplus \Theta(\alpha_{2n+1}),
  \end{align}
where in both cases $\Theta(\alpha_j)$ acts on $\ell^2(\{j, j+1\})$ and $\Theta$ is given by 
 \begin{equation} \Theta(\alpha) 
= \begin{bmatrix} \overline{\alpha} & \sqrt{1-|\alpha|^2}\ \\ \sqrt{1-|\alpha|^2} & -\alpha \end{bmatrix}. \end{equation}

If $\alpha$ is periodic of period $N$ and $N$ is even, then we consider the Floquet matrices $\calF_N(\theta)$ given by restricting to $[0,N-1]$ with the boundary condition $u_{n+N}=e^{i\theta}u_n$. 

One can check that the $\calL\calM$ factorization induces a corresponding factorization of the  Floquet operators. That is, with
\begin{align}
\calL_N(\theta) & =  \begin{bmatrix} \Theta(\alpha_0) \\ & \Theta(\alpha_2) \\ && \ddots \\ &&& \Theta(\alpha_{N-2}) \end{bmatrix} \\
\calM_N(\theta) & = \begin{bmatrix} -\alpha_{N-1} &&&& e^{-i\theta}\rho_{N-1} \\ & \Theta(\alpha_1) \\ && \ddots \\ &&& \Theta(\alpha_{N-3}) \\ e^{i\theta}\rho_{N-1} &&&& \overline{\alpha}_{N-1} \end{bmatrix},
\end{align}
we have
\begin{equation}
\calF_N(\theta) = \calL_N(\theta)\calM_N(\theta).
\end{equation}
Since we are interested in computations, let us write out the exact form of $\calF_N(\theta)$ for relevant ranges of $N \in 2 \bbN$.
For $N=2$, 
\begin{align}
\nonumber
\calF_N(\theta)
& = \calL_N(\theta) \calM_N(\theta) \\
\nonumber
& = \begin{bmatrix} \overline{\alpha_0} & \rho_0 \\ {\rho_0} & -\alpha_0 \end{bmatrix}\begin{bmatrix} -\alpha_1 & e^{-i\theta}\rho_1 \\  e^{i\theta} \rho_1 & \overline{\alpha}_1 \end{bmatrix} \\
\label{eq:floqCMVper2}
& = \begin{bmatrix} -\alpha_1\overline{\alpha_0}+e^{i\theta}\rho_1\rho_0 
& e^{-i\theta}\rho_1\overline{\alpha_0} + \overline{\alpha_1}\rho_0 \\
-\alpha_1 \rho_0 - e^{i\theta}\rho_1\alpha_0
& e^{-i\theta}\rho_1 \rho_0 -\overline{\alpha_1}\alpha_0 \end{bmatrix}.
\end{align}
Similarly, for $N=4$, we have
\begin{equation}
\label{eq:floqCMVper4}
\calF_N(\theta)
= \begin{bmatrix}
-\overline{\alpha_0}\alpha_{3} 
& \overline{\alpha_1}\rho_0 & \rho_1\rho_0 
& {e^{-i\theta}\overline{\alpha_0}\rho_{3}}\\
-{\rho_0}\alpha_{3} 
& {-\overline{\alpha_1}\alpha_0} 
& -\rho_1 \alpha_0 
& {e^{-i\theta}{\rho_0\rho_{3}}} 
 \\
e^{i\theta} \rho_{3}\rho_{2} 
& \overline{\alpha_{2}}\rho_{1} & -\overline{\alpha_{2}}\alpha_{1} & \overline{\alpha_{3}} \rho_{2}  \\
-e^{i\theta} \rho_{3}\alpha_{2}  
& {\rho_{2}\rho_{1}} & -{\rho_{2}}\alpha_{1} & -\overline{\alpha_{3}}\alpha_{2}   
\end{bmatrix}.
\end{equation}

In general, for $N \geq 6$, these have the form

\begin{equation}\label{eq:floqCMVperGeq6}
\tiny
\calF_N(\theta)
=
\begin{bmatrix}
-\overline{\alpha_0}\alpha_{N-1} 
& \overline{\alpha_1}\rho_0 & \rho_1\rho_0 
&&& && {e^{-i\theta}\overline{\alpha_0}\rho_{N-1}}\\
-{\rho_0}\alpha_{N-1} 
& {-\overline{\alpha_1}\alpha_0} 
& -\rho_1 \alpha_0 
&&&&& {e^{-i\theta}{\rho_0\rho_{N-1}}} 
 \\ 
& \overline{\alpha_2}\rho_1 & -\overline{\alpha_2}\alpha_1 & \overline{\alpha_3} \rho_2 & \rho_3\rho_2 & \\
& {\rho_2\rho_1} & -{\rho_2}\alpha_1 & -\overline{\alpha_3}\alpha_2 & -\rho_3\alpha_2 &   \\ \\
&& \ddots  & \ddots  & \ddots  & \ddots \\ \\
&&& \overline{\alpha_{N-4}}\rho_{N-5} & -\overline{\alpha_{N-4}}\alpha_{N-5} & \overline{\alpha_{N-3}} \rho_{N-4} & \rho_{N-3}\rho_{N-4} & \\
&&& {\rho_{N-4}\rho_{N-5}} & -{\rho_{N-4}}\alpha_{N-5} & -\overline{\alpha_{N-3}}\alpha_{N-4} & -\rho_{N-3}\alpha_{N-4} &    \\
e^{i\theta} \rho_{N-1}\rho_{N-2} &&&&& \overline{\alpha_{N-2}}\rho_{N-3} & -\overline{\alpha_{N-2}}\alpha_{N-3} & \overline{\alpha_{N-1}} \rho_{N-2}  \\
-e^{i\theta} \rho_{N-1}\alpha_{N-2}  &&&&& {\rho_{N-2}\rho_{N-3}} & -{\rho_{N-2}}\alpha_{N-3} & -\overline{\alpha_{N-1}}\alpha_{N-2}   
\end{bmatrix}.
\end{equation}

Let us also define the Szeg\H{o} transfer matrices. For $z \in \bbC$ and $\alpha \in \bbD$, one defines
\begin{equation}\label{eq:Salphaz:def}
S(\alpha,z) = \frac{1}{\sqrt{1-|\alpha|^2}} \begin{bmatrix}  z & -\overline\alpha \\ -\alpha z & 1\end{bmatrix}.
\end{equation}
For our numerical calculations, we want to note the following fact, which relates zeros of the trace of a product of Szeg\H{o} matrices to eigenvalues of a suitable Floquet cutoff and follows from the general theory of periodic CMV matrices. For completeness, we include the short proof.

\begin{prop} \label{prop:DNzerosAreEigs}
Suppose $\{\alpha_n\}_{n \in \bbZ}$ is $N$-periodic, let $\Delta_N$ denote the associated discriminant given by
\begin{equation} \label{eq:normalizedDNz:def}
\Delta_N(z) = \tr(z^{-N/2}S(\alpha_N,z)S(\alpha_{N-1},z) \cdots S(\alpha_1,z)),
\end{equation} 
and consider the Floquet matrices as in \eqref{eq:floqCMVper2}, \eqref{eq:floqCMVper4}, and \eqref{eq:floqCMVperGeq6}. 
\begin{enumerate}
\item[{\rm(a)}] If $N$ is even, then $z$ is a zero of $\Delta_N$ if and only if $z$ is an eigenvalue of $\calF_N(\pi/2)$.
\item[{\rm(b)}] If $N$ is odd, then $z$ is a zero of $\Delta_N$ if and only if $z$ is an eigenvalue of $\calF_{2N}(\pi)$.
\end{enumerate}
\end{prop}

\begin{proof}
Suppose $N$ is even.
Setting $\theta = \pi/2$ in \cite[Eq.~(11.2.17)]{Simon2005:OPUC2} (notice that $\beta = e^{i\theta}$ in Simon's notation) gives 
\begin{equation} \label{eq:fromdetFNtoTraceDeltaN}
\det(z-\calF_N(\pi/2)) = z^{N/2}\left[\prod_{j=0}^{N-1} \rho_j \right] \Delta_N(z).
\end{equation}
In view of \eqref{eq:normalizedDNz:def}, this implies that the zeros of $\Delta_N(z)$ and $\det(z-\calF_N(\pi/2))$ coincide, proving (a). The proof of (b) follows in a similar fashion by using \cite[Eq.~(11.2.17)]{Simon2005:OPUC2} with $\theta = \pi$ to get 
\begin{equation} 
\det(z-\calF_{2N}(\pi)) = z^{N}\left[\prod_{j=0}^{2N-1} \rho_j \right] (\Delta_{2N}(z)+2)
\end{equation}
together with the identity
\[A^2 = \tr(A)A-I\]
for $A \in \SL(2,\bbC)$, which implies $\Delta_{2N}(z) = [\Delta_N(z)]^2 - 2$.
\end{proof}

We mention this connection for two reasons. First, the results that one brings together to connect the Ising partition function to gap labels naturally relate to the two sides of \eqref{eq:fromdetFNtoTraceDeltaN}. More specifically, the gap labelling theorem that we will formulate in Theorem~\ref{t:CMVgl} concerns the density of states measure, which is related to normalized eigenvalue counting measures associated with cutoff operators in a natural way, and hence connects to the left hand side of \eqref{eq:fromdetFNtoTraceDeltaN}, whereas Theorem~\ref{prop:IsingzerosToDN} gives a connection between $\Delta_N(z)$, which appears on the right hand side of \eqref{eq:fromdetFNtoTraceDeltaN}, and the partition function of an associated Ising model. Secondly, finding roots of polynomials can be numerically delicate (depending on the basis in which the polynomials are expressed, the magnitude of the coefficients, and the algorithm for finding roots), whereas computing eigenvalues of unitary matrices is robust.  Indeed, a common way to compute roots of polynomials expressed in the monomial basis is to find eigenvalues of the associated companion matrix, which can yield poor results~\cite{TT94}.

\subsection{Ergodic CMV Matrices}

\begin{definition}
Let $(\Omega,T,\mu)$ denote an ergodic topological dynamical system as in Section~\ref{ssec:topdyn}, that is, $\Omega$ is a compact metric space, $T:\Omega \to \Omega$ is a homeomorphism, and $\mu$ is a $T$-ergodic probability measure. Given a continuous function $f:\Omega \to \bbD$, the associated \emph{ergodic family of CMV matrices} is $\{\calE(\omega)\}_{\omega \in \Omega}$, where $\calE(\omega)$ is defined by the coefficients
\begin{equation} \label{eq:ergCMVAlphaNDef}
\alpha_n(\omega)= f(T^n\omega), \quad \omega \in \Omega, \ n \in \bbZ.
\end{equation}
\end{definition}

\begin{definition}
For each $N \in \bbN$, $\omega \in \Omega$, we define the measure $\kappa_{\omega,N}$ to be the normalized eigenvalue counting measure of $\calE(\omega)\chi_{[0,N-1]}$, that is,
\begin{equation}
\int f \, d\kappa_{\omega,N} = \frac{1}{N} \tr\, f(\calE(\omega)\chi_{[0,N-1]}).
\end{equation}
We also define the \emph{density of states (DOS) measure} $\kappa$ by
\begin{equation}
\int_{\partial \bbD}g \, d\kappa = \int_\Omega \langle \delta_0, g(\calE(\omega))\delta_0 \rangle \, d\mu(\omega),
\end{equation}
and note that $\kappa_{\omega,N} \to \kappa$ weakly as $N \to \infty$ for $\mu$-a.e.\ $\omega \in \Omega$ by arguments using ergodicity~\cite[Theorem~10.5.21]{Simon2005:OPUC2}.
\end{definition}

Let us see that one can recover the DOS from the zeros of the discriminants associated with periodic operators defined by the ergodic family. Given $\calE(\omega)$ as above, define
\[ D_{N}(z,\omega) = \tr\left[ S(\alpha_{N}(\omega),z) \cdots S(\alpha_1(\omega),z)\right]. \]
It is known that $D_{N}(\cdot, \omega)$ has $N$ distinct zeros $\xi_1(\omega),\ldots, \xi_{N}(\omega)$ that lie on $\partial \bbD$ \cite{Simon2005:OPUC2}. We denote by $\nu_{\omega,N}$ the normalized zero-counting measure, that is,
\begin{equation}
\int_{\partial \bbD} f \, d\nu_{\omega,N} = 
\frac{1}{N} \sum_{n=1}^{N} f(\xi_n(\omega)).
\end{equation}

\begin{prop} \label{prop:DOSviaFloquet}
With notation as above, one has $\nu_{\omega,N} \to \kappa$ weakly for a.e.\ $\omega \in \Omega$.
\end{prop}

\begin{proof}
From \cite[Theorem~10.5.21]{Simon2005:OPUC2}, we know $\kappa_{\omega,N} \to \kappa$, so it suffices to show
that $\nu_{\omega,N}$ has the same weak limit as $\kappa_{\omega,N}$. Using Proposition~\ref{prop:DNzerosAreEigs}, we see that $\nu_{\omega,N}$ is the normalized eigenvalue counting measure of a suitable Floquet cutoff of $\calE(\omega)$, so the desired conclusion holds by a direct calculation.
\end{proof}

\section{From Ising, Lee, and Yang to Cantero, Moral, and Vel\'{a}zquez} \label{sec:LY2Sz}

Let us explain how the relationship between Lee--Yang zeros and discriminants of CMV matrices arises.  The first observation is that one can characterize the partition function $\calZ_N(\zeta)$ as the trace of a suitable matrix product. This matrix formalism is well-known to experts, but we include a detailed discussion for ease of reading. To define the aforementioned matrix product, write
\begin{equation} \label{eq:isingTMdef}
M(\beta,\zeta) = 
\begin{bmatrix} \beta\zeta & 1/\beta  \\[2mm] 1/\beta  & \beta/\zeta \end{bmatrix}
\end{equation}
for $\beta,\zeta \in \bbC \setminus\{0\}$.

\begin{prop} \label{prop:isingTMs}
Let $p_n>0$ be given for $1 \le n \le N$, define $\beta_n$ as in \eqref{eq:wbetadef}, and let $\calZ_N$ denote the associated partition function. One has
\begin{equation}
\calZ_N(\zeta) = \tr[M(\beta_N,\zeta)M(\beta_{N-1}, \zeta) \cdots M(\beta_1, \zeta)]
\end{equation}
for all $\zeta \neq 0$, where $M$ is given by \eqref{eq:isingTMdef}.
\end{prop}

\begin{proof}
The proof of this result can be found in most standard textbooks on solvable models in statistical mechanics, e.g., \cite{Baxter1989}. We reproduce the proof of \cite{Baxter1989} here to keep the paper more self-contained.

Write the entries of a $2 \times 2$ matrix as
\begin{equation} \label{eq:Asigmasigma'def}
A = \begin{bmatrix} A_{+,+} & A_{+,-} \\ A_{-,+} & A_{-,-} \end{bmatrix}.
\end{equation}
In particular, combining \eqref{eq:isingTMdef} and \eqref{eq:Asigmasigma'def} gives
\begin{equation}
M(\beta, \zeta)_{\sigma,\sigma'} = \beta^{\sigma\sigma'} \zeta^{(\sigma+\sigma')/2}.
\end{equation}
Consequently, using \eqref{eq:calZNdef} and \eqref{eq:sigmaperbc}, we have
\begin{align*}
\calZ_N(\zeta) 
& = \sum_{\sigma \in \Lambda_N} \prod_{n=1}^N \beta_n^{\sigma_n\sigma_{n+1}} \zeta^{\sigma_n}  \\
& = \sum_{\sigma \in \Lambda_N} \prod_{n=1}^N \beta_n^{\sigma_n\sigma_{n+1}} \zeta^{(\sigma_n+\sigma_{n+1})/2} \\
& = \sum_{\sigma \in \Lambda_N} \prod_{n=1}^N M(\beta_n,\zeta)_{\sigma_n,\sigma_{n+1}}.
\end{align*}
Now split the sum and use the periodic boundary condition again to get
\begin{align*}
\calZ_N(\zeta) 
& = \sum_{\sigma \in \Lambda_N} \prod_{n=1}^N M(\beta_n, \zeta)_{\sigma_n,\sigma_{n+1}} \\
& = \sum_{\sigma_1 \in \Lambda_1} \sum_{(\sigma_2,\ldots,\sigma_N) \in \Lambda_{N-1}} \prod_{n=1}^N M(\beta_n, \zeta)_{\sigma_n,\sigma_{n+1}} \\
& = \sum_{\sigma_1 \in \Lambda_1} [M(\beta_1,\zeta) \cdots M(\beta_N,\zeta)]_{\sigma_1,\sigma_1} \\
& = \tr[M(\beta_1, \zeta) \cdots M(\beta_N, \zeta)].
\end{align*}
The result follows by noting that $\beta_n>0$, so $M$ is symmetric and thus one can reverse the order of the factors by taking the transpose.
\end{proof}

With Proposition~\ref{prop:isingTMs} proved, let us now connect back to CMV matrices, by way of the Szeg\H{o} transfer matrices introduced in \eqref{eq:Salphaz:def}. 

\begin{prop} \label{prop:IsingzerosToDN}
With notation as in Proposition~\ref{prop:isingTMs}, the zeros of $\calZ_N(\zeta)$ are the same as the zeros of $D_N(\zeta^2)$, where
\begin{equation}
D_N(z) := \tr\left[S(\beta_N^{-2},z) \cdots S(\beta_1^{-2},z)\right].
\end{equation}
Equivalently, the zeros of $\calZ_N(\zeta)$ coincide with those of $\widetilde D_N(\zeta)$, where
\begin{equation}
\widetilde D_N(z) := \tr\left[S(\beta_N^{-2},z)S(0,z)S(\beta_{N-1}^{-2},z)S(0,z)\cdots S(\beta_1^{-2},z) S(0,z) \right].
\end{equation}
\end{prop}

\begin{proof}
For $\zeta \neq 0$ and $\beta>1$, notice that
\begin{align*}
\begin{bmatrix} -1 & 0 \\ 0 & \zeta\end{bmatrix}
M(\beta,\zeta)
\begin{bmatrix} -1 & 0 \\ 0 & 1/\zeta\end{bmatrix}
& = \begin{bmatrix} -1 & 0 \\ 0 & \zeta\end{bmatrix}
\begin{bmatrix} \beta\zeta & \frac{1}{\beta} \\[2mm] \frac{1}{\beta} & \frac{\beta}{\zeta}\end{bmatrix}
\begin{bmatrix} -1 & 0 \\ 0 & 1/\zeta\end{bmatrix} \\
& = \begin{bmatrix} \beta\zeta & -\frac{1}{\beta\zeta} \\[2mm] -\frac{\zeta}{\beta} & \frac{\beta}{\zeta}\end{bmatrix} \\
& = \frac{\beta}{\zeta} \begin{bmatrix} \zeta^2 & -\frac{1}{\beta^2} \\[2mm] -\frac{\zeta^2}{\beta^2} & 1\end{bmatrix}  \\
& = \frac{\beta}{\zeta}\sqrt{1-\beta^{-2}}\,S(\beta^{-2},\zeta^2).
\end{align*}
Thus, $S(\beta_N^{-2}, \zeta^2) \cdots S(\beta_1^{-2}, \zeta^2)$ is similar to a nonzero multiple of $M(\beta_N,\zeta)\cdots M(\beta_1,\zeta)$ and the first claim follows from Proposition~\ref{prop:isingTMs}. The second follows follows from noting that
\[S(\alpha,z) S(0,z)= S(\alpha,z^2)\]
for any $\alpha \in \bbD$, $z \in \bbC$.
\end{proof}

\section{From Szeg\H{o} to Schwartzman \`a la Geronimo and Johnson} \label{sec:Sz2Schwartman}

In the previous section, we saw how to relate the partition functions of a ferromagnetic Ising model to discriminants associated with a family of periodic CMV matrices. Next, we want to explain how to relate zeros of the discriminant to the rotation number of the Szeg\H{o} cocycle and hence to the range of the Schwartzman homomorphism. Throughout this section we fix an ergodic topological dynamical system $(\Omega,T,\mu)$, a continuous $f:\Omega \to \bbD$, and let $\{\calE(\omega)\}$ denote the associated family of CMV matrices defined by \eqref{eq:ergCMVAlphaNDef}. We assume $\supp\mu = \Omega$.

By general arguments, there exists a compact set $\Sigma \subseteq \partial \bbD$ such that $\sigma(\calE(\omega))=\Sigma$ for $\mu$-a.e.\ $\omega \in \Omega$. Moreover, $\sigma(\calE(\omega))=\Sigma$ for any $\omega$ with a dense $T$-orbit.

Given this setup, we define
\[A_z(\omega) = z^{-1}S(f(T\omega),z)S(f(\omega),z), \quad \omega \in \Omega, \ z \in \bbC \setminus\{0\},\]
where $S$ is given by \eqref{eq:Salphaz:def}. We can then characterize the almost-sure spectrum $\Sigma$ as the complement of the set where $(T^2,A_z)$ is uniformly hyperbolic.

\begin{theorem}[Johnson's Theorem for CMV matrices] \label{t:cmv:johnson}
Assume $(\Omega,T,\mu)$ is an ergodic topological dynamical system such that $\supp \mu = \Omega$, $f \in C(\Omega,\bbD)$, and let $\Sigma$ denote the associated almost-sure spectrum associated with the family $\{\calE(\omega)\}_{\omega \in \Omega}$. We have
\begin{equation}
\partial \bbD \setminus \Sigma = \UH:=\{z \in \partial \bbD : (T^2,A_z) \text{ is uniformly hyperbolic}\}
\end{equation}
\end{theorem}

See \cite{DFLY2016DCDS} for additional details and a proof. One major application of Theorem~\ref{t:cmv:johnson} is the gap labelling theorem for ergodic CMV matrices, which we formulate presently.

\begin{theorem} \label{t:CMVgl}
Let $(\Omega,T,\mu)$ denote an ergodic topological dynamical system such that $\Omega = \supp \mu$, $f \in C(\Omega,\bbD)$, and $\{\calE(\omega)\}_{\omega \in \Omega}$ the associated ergodic family of extended CMV matrices. Let $\kappa$ and $\Sigma$ denote the density of states measure and almost-sure spectrum associated with this family. For any $z_1,z_2 \in \partial \bbD \setminus  \Sigma$, one has
\begin{equation}
\kappa([z_1,z_2]) \in \frA(\Omega,T,\mu),
\end{equation}
where $[z_1,z_2]$ denotes the closed arc from $z_1$ to $z_2$ in the counterclockwise direction.
\end{theorem}

\begin{proof}
This result is a consequence of \cite{GeronimoJohnson1996JDE} and \cite{Simon2005:OPUC1}. Let $\rho$ denote the rotation number associated with the family $\{\calE(\omega)\}$, as defined in \cite[Section~4]{GeronimoJohnson1996JDE}; notice that this comes with a factor $\frac{1}{2}$ (cf.~\cite[Eq.~(4.9)]{GeronimoJohnson1996JDE}). On the one hand, \cite[Theorem~5.6]{GeronimoJohnson1996JDE} asserts that $\rho$ takes values in ($2\pi$ times) the Schwartzman group in the gaps of $\Sigma$ (that is, on connected components of $\partial \bbD \setminus \Sigma$. Notice that the Schwartzman group in \cite{GeronimoJohnson1996JDE} differs from ours by a factor of $2\pi$; compare the first displayed equation on \cite[p.~171]{GeronimoJohnson1996JDE}. On the other hand, $\rho$ is related to the Lyaupnov exponent via \cite[Theorem~4.7]{GeronimoJohnson1996JDE}. Namely, there is an analytic function $w(z)$ such that the boundary values of $\Re\, w$ give the Lyapunov exponent and the boundary values of $\Im\, w$ give $\rho$. By combining this with \cite[Theorems~10.5.8 and~10.5.21]{Simon2005:OPUC2}, we are done.
\end{proof}

The main result follows by combining all of these pieces.

\begin{proof}[Proof of Theorem~\ref{t:isingGL}]
This follows by combining Proposition~\ref{prop:DOSviaFloquet}, Proposition~\ref{prop:IsingzerosToDN} and Theorem~\ref{t:CMVgl}. More specifically, Propositions~\ref{prop:DOSviaFloquet} and \ref{prop:IsingzerosToDN} show that the left-hand side of \eqref{eq:isingGL} is $\kappa([z_1,z_2])$ and Theorem~\ref{t:CMVgl} shows that this quantity belongs to $\frA$.
\end{proof} 

\section{A Gallery of Lee--Yang and CMV Zeros} \label{sec:gallery}

Let us conclude with a discussion of several examples, including some plots of the distributions of the relevant zeros. By combining Propositions~\ref{prop:IsingzerosToDN} and~\ref{prop:DNzerosAreEigs}, the zeros of $\calZ_N$ may be computed by finding eigenvalues of $\calF_{N}(\pi/2)$, where $\calF_{N}$ denotes the Floquet matrices associated with a suitable extended CMV matrix, which is the approach employed in the numerics below.

Before embarking on these computations, we note one potential opportunity for acceleration of the 
eigenvalue calculation for large-scale problems.  The corner entries in~\eqref{eq:floqCMVperGeq6}
cause the matrices $\calF_N(\theta)$ to have full bandwidth.  
In the case of Jacobi matrices, a simple reordering described in~\cite{PEF2015IEOT} results
in Hermitian matrices of bandwidth~5, allowing for efficient numerical eigenvalue computations.
An analogous reordering is possible here, though the CMV structure makes this reordering
a bit more intricate; the result is a matrix having bandwidth~9.
Such structure is less clearly exploitable in non-Hermitian eigenvalue computations,
but QR-related algorithms for CMV matrices (see~\cite{BE91,VW12}) could 
potentially be adapted and extended to this case.

We consider a CMV matrix of even period $N$, and define a permutation $p:\{1,\ldots,N\} \to \{1,\ldots,N\}$ by
\begin{equation}
p(j) = 
\begin{cases}
2j-1, & \mbox{$j$ odd and $1 \le j < N/2$;} \\
2N+1-2j, & \mbox{$j$ odd and $N/2 < j < N$;}\\
2j, & \mbox{$j$ even and $1 < j \le N/2$;} \\
2N+2-2j, & \mbox{$j$ even and $N/2 < j \leq N$.}
\end{cases}
\end{equation}
Let $P$ denote the associated permutation matrix, which is given by $Pe_j = e_{p(j)}$. Equivalently,
\begin{equation}
P_{i,j} = \delta_{i,p(j)}
\end{equation} 
The bandwidth of the reordered matrix $P\calF_N(\theta)P^*$ is at most $9$.

\begin{prop} \label{prop:bandwidth}
Suppose $N \geq 2$ is even. The bandwidth of $\widetilde{\calF}_N(\theta) := P\calF_N(\theta)P^*$ is at most $9$. More precisely,
\begin{equation}
\langle e_j, \widetilde \calF_N(\theta)e_k \rangle = 0
\end{equation} 
whenever $|j-k|>4$.
\end{prop}

\begin{proof}
If $N=2$ or $4$, the claim holds vacuously, so assume $N \geq 6$.
Write $d_N(j,k) = \min\{|j-k|,N-|j-k|\}$, and notice that $\langle e_j, \calF_N(\theta)e_k\rangle=0$ if $d_N(j,k)>2$. Since
\begin{equation}
\langle e_j,\widetilde{\calF}_N(\theta)e_k \rangle = 
\langle e_{p^{-1}(j)}, \calF_N(\theta)e_{p^{-1}(k)} \rangle,
\end{equation} 
it suffices to demonstrate 
\begin{equation}\label{eq:bandwidthgoal}
  |j-k|>4 \implies d_N(p^{-1}(j),p^{-1}(k))>2.
\end{equation}
Equivalently, we may show
\begin{equation}\label{eq:bandwidthgoal2}
  d_N(\ell,m)\leq 2 \implies |p(\ell) - p(m)| \leq 4.
\end{equation}
It is straightforward (albeit a little tedious) to verify \eqref{eq:bandwidthgoal2} from the definition of $p$. Indeed, if $\ell = m+1$, we have (using $N+1=1$)
\[p(m+1) - p(m) = \begin{cases} 
2(m+1)-(2m-1) = 3, & \mbox{$m$ odd and  $m<N/2$;} \\
2(m+1)-1-2m = 1, & \mbox{$m$ even and  $m<N/2$;} \\ 
2N+2-2(m+1) - (2N+1-2m) = -1, & \mbox{$m$ odd and  $N/2<m<N$;} \\ 
2N+1-2(m+1) - (2N+2-2m) = -3, & \mbox{$m$ even and $N/2<m<N$;} \\
1-2 = -1, & \mbox{$m=N$.} \end{cases}\]
The results are similar, but slightly more laborious for $p(m+2)-p(m)$.
\end{proof}

For inspiration and context, here is a picture of the permutation when $N=24$. 
\begin{center}
\includegraphics[scale=.5]{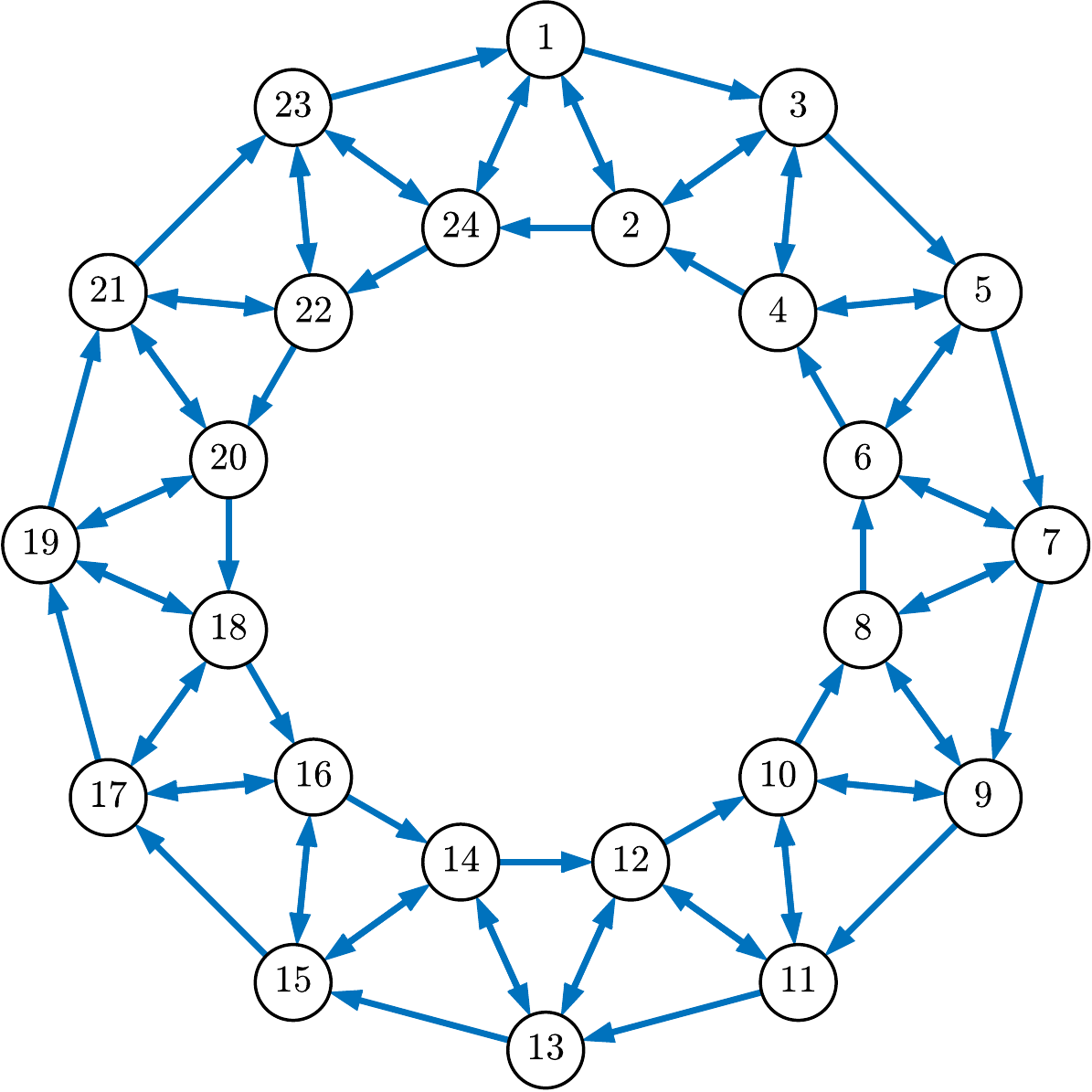}
\qquad
\includegraphics[scale=.5]{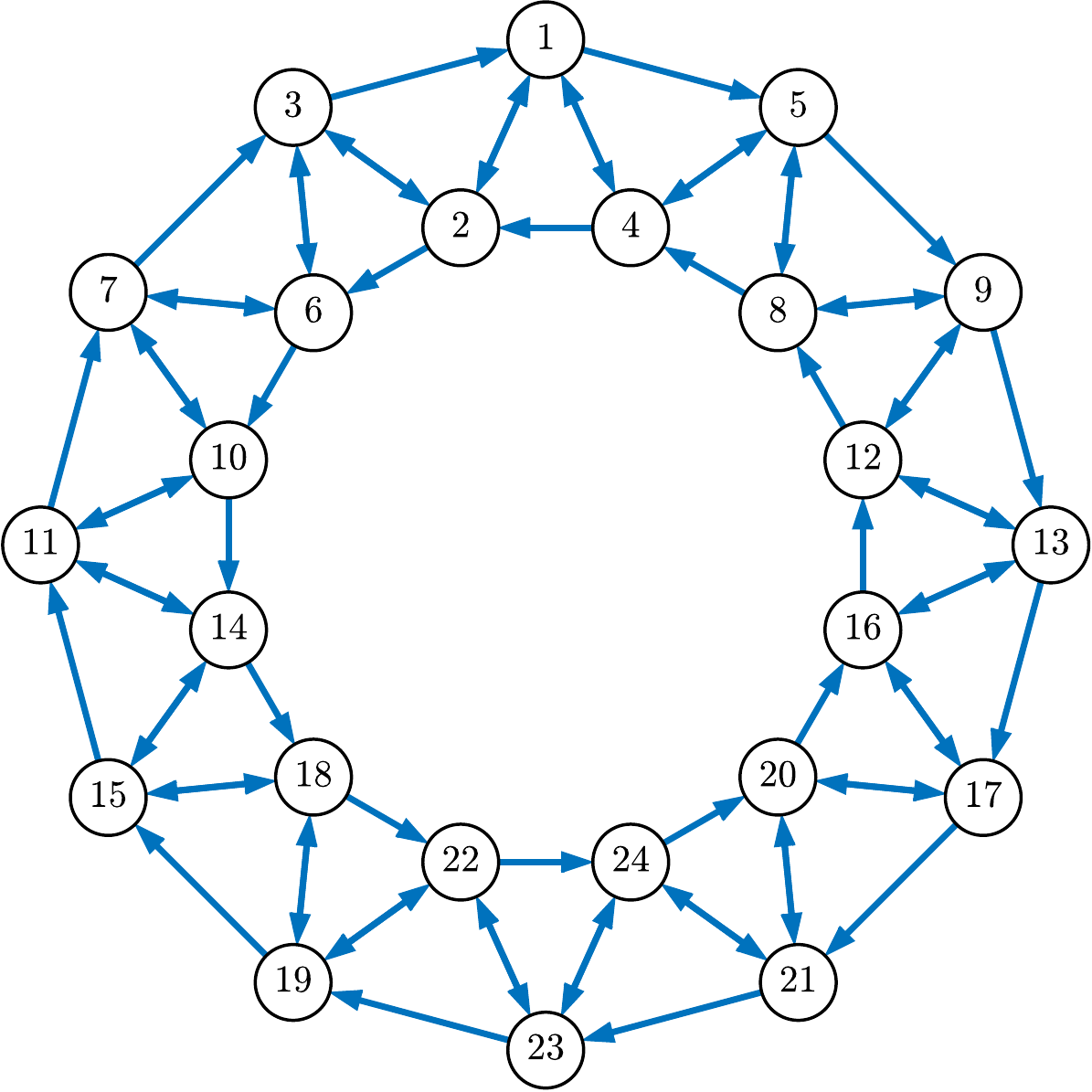}
\end{center}
The graph on the left shows the adjacency relations for $\calF_N(\theta)$, having an edge from $i$ to $j$ if $i \neq j$ and $[\calF_N(\theta)]_{ij} \neq 0$ (for a generic CMV matrix). The graph on the right shows the same scheme for $\widetilde{\calF}_N(\theta)$. From this perspective, the conclusion of Proposition~\ref{prop:bandwidth} is easy to check visually: one simply verifies that the indices of connected nodes can differ by no more than four in the adjacency graph for $\widetilde{\calF}_N(\theta)$.

The figures below show the corresponding nonzero pattern of the matrix $\calF_N(\theta)$ (left) and its reordered version $\widetilde{\calF}_N(\theta)$ (right).

\begin{center}
\includegraphics[width=2in]{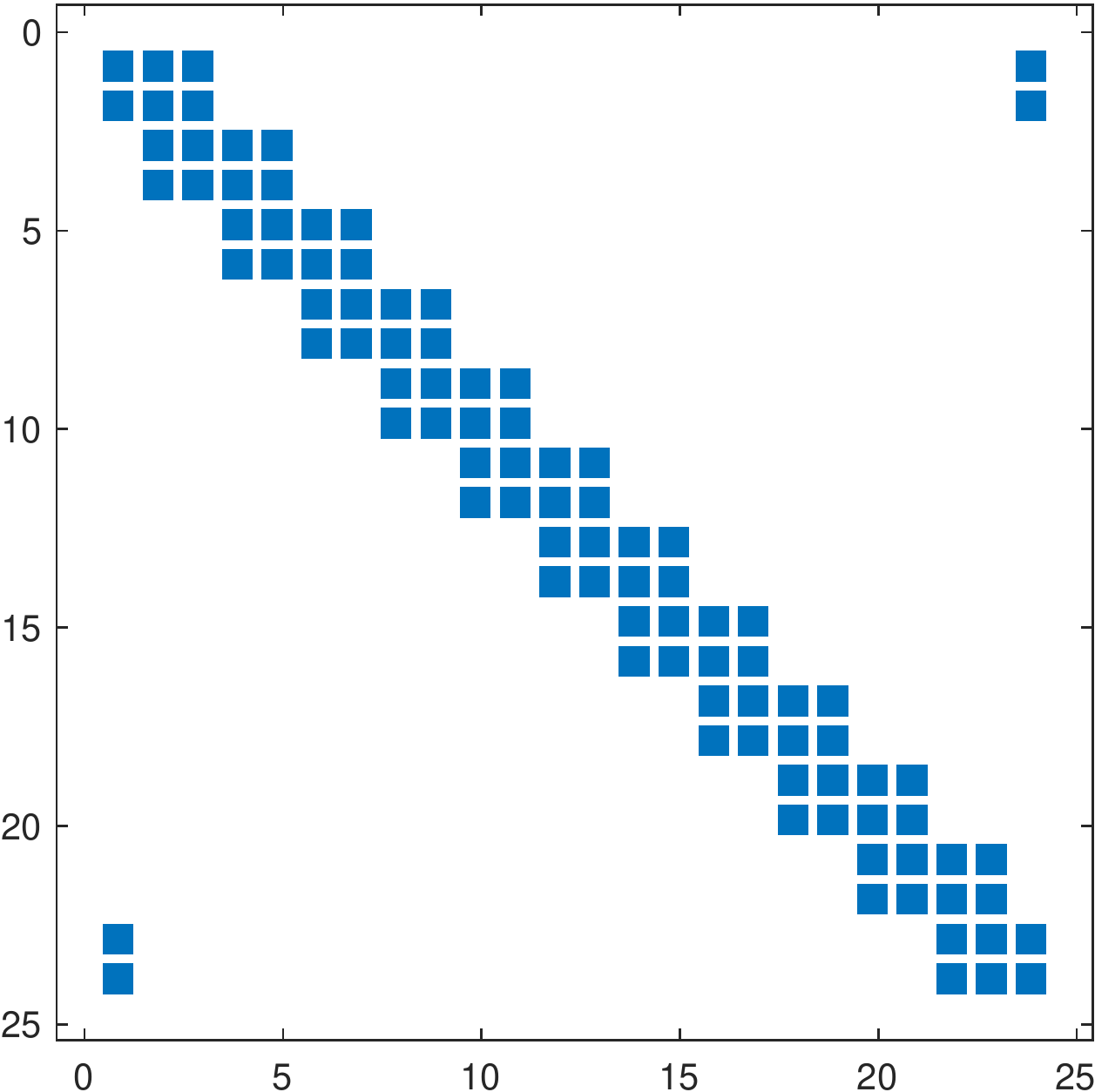}
\qquad\qquad
\includegraphics[width=2in]{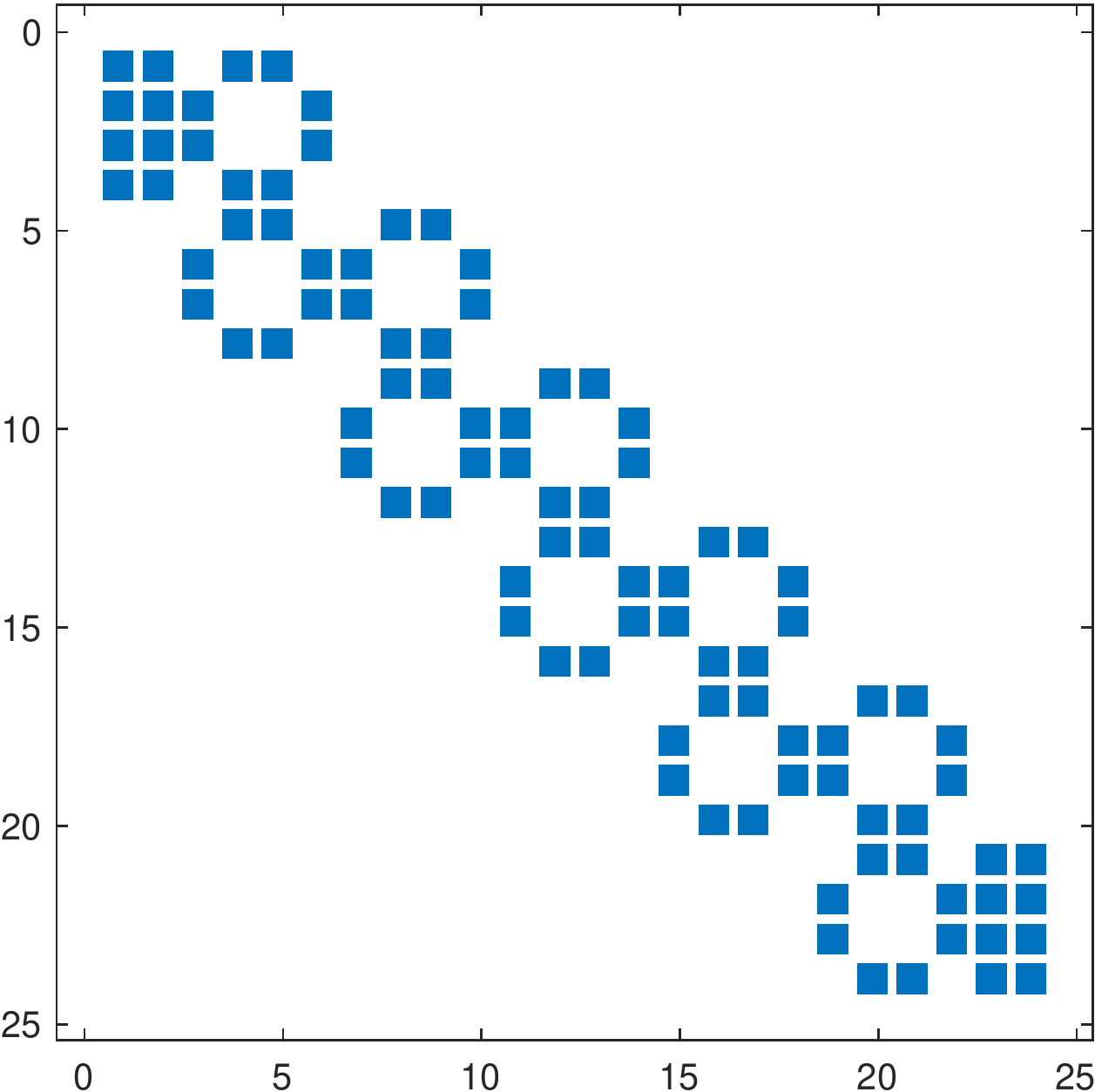}
\end{center}

Let us conclude with some plots of zeros and numerical approximations of the density of zeros. We begin with the Fibonacci case, which supplied the original motivation for the present work.

\begin{example}[Fibonacci]
Consider an alphabet $\calA = \{\sfa,\sfb\}$ with two letters, and let $\calA^*$ denote the free monoid over $\calA$ (that is, the set of finite words written with the letters in $\calA$). The \emph{Fibonacci substitution} is defined by $S(\sfa) = \sfa\sfb$ and $S(\sfb) = \sfa$, and extended to $\calA^*$ by concatenation. Thus, beginning with the seed $u_0=\sfa$, one forms the sequence $u_k = S^k(w_0)$:
\begin{align*}
u_1 & = \sfa\sfb \\
u_2 & = \sfa\sfb\sfa \\
u_3 & = \sfa\sfb\sfa\sfa\sfb \\
u_4 & = \sfa\sfb\sfa\sfa\sfb\sfa\sfb\sfa, 
\end{align*}
and so on. As one can see, the initial letters stabilize once they appear, so one can define the word $u_\infty = \lim_{k \to \infty} u_k = \sfa\sfb\sfa\sfa\sfb\sfa\sfb\sfa\ldots$, where the limit may be understood in the sense of the product topology on $\calA^\bbN$ (and in which $\calA$ has the discrete topology). 

One can specify a ferromagnetic Ising model by choosing $p_\sfa, p_\sfb>0$ and defining the sequence of normalized magnetic couplings via
\begin{equation}
p_n = \begin{cases} p_\sfa, & u_\infty(n)=\sfa;\\  p_\sfb, & u_\infty(n)=\sfb. \end{cases}
\end{equation}

There are several equivalent ways to imbed this example into an ergodic context. The \emph{Fibonacci subshift}, $\Omega_{\rm F}\subseteq \calA^\bbZ$ is the set of all sequences whose local structure coincides with that of $u_\infty$. More precisely, if $v =v_1\cdots v_\ell \in \calA^*$ and $u$ is a finite word or infinite sequence, we write $v\triangleleft u$ if for some $j$, $v = u(j)u(j+1) \cdots u(j+\ell-1)$ (and we say $v$ is a \emph{subword} of $u$). One then defines
\begin{equation}
\Omega_{\rm F} = \{ \omega = (\omega_n)_{n\in\bbZ} : \forall \ell \in \bbN, \ n \in \bbZ,  \omega_n \cdots \omega_{n+\ell-1} \triangleleft u_\infty\}.
\end{equation}
One can check that $\Omega_{\rm F}$ is a compact subset of $\calA^\bbZ$ that is invariant under the action of the shift $[T\omega]_n = \omega_{n+1}$. It is furthermore known that $(\Omega_{\rm F},T)$ enjoys a unique $T$-invariant measure $\mu$ satisfying $\supp\mu = \Omega_{\rm F}$.

For this system, the set of labels can be computed explicitly:
\begin{equation}
\frA(\Omega_{\rm F}, T, \mu) = \bbZ + \alpha \bbZ = \{n+m\alpha : n,m \in \bbZ\},
\end{equation}
where $\alpha = (\sqrt{5}-1)/2$ denotes the inverse of the golden mean; see, e.g., \cite{BelBovGhe1992, DFGap} for details.

Let us show some plots for this model. Following~\cite{BaaGriPis1995JSP}, we take $p_\sfa = 2/3$ and $p_\sfb = 1/100$. First, we show the zeros of the partition function for the Ising model corresponding to $n=10$ and $n=17$ iterations of the Fibonacci substitution.  
(The $n=10$ plot replicates the analogous plot in~\cite[Fig.~1]{BaaGriPis1995JSP}.)
\begin{center}
\includegraphics[scale=.45]{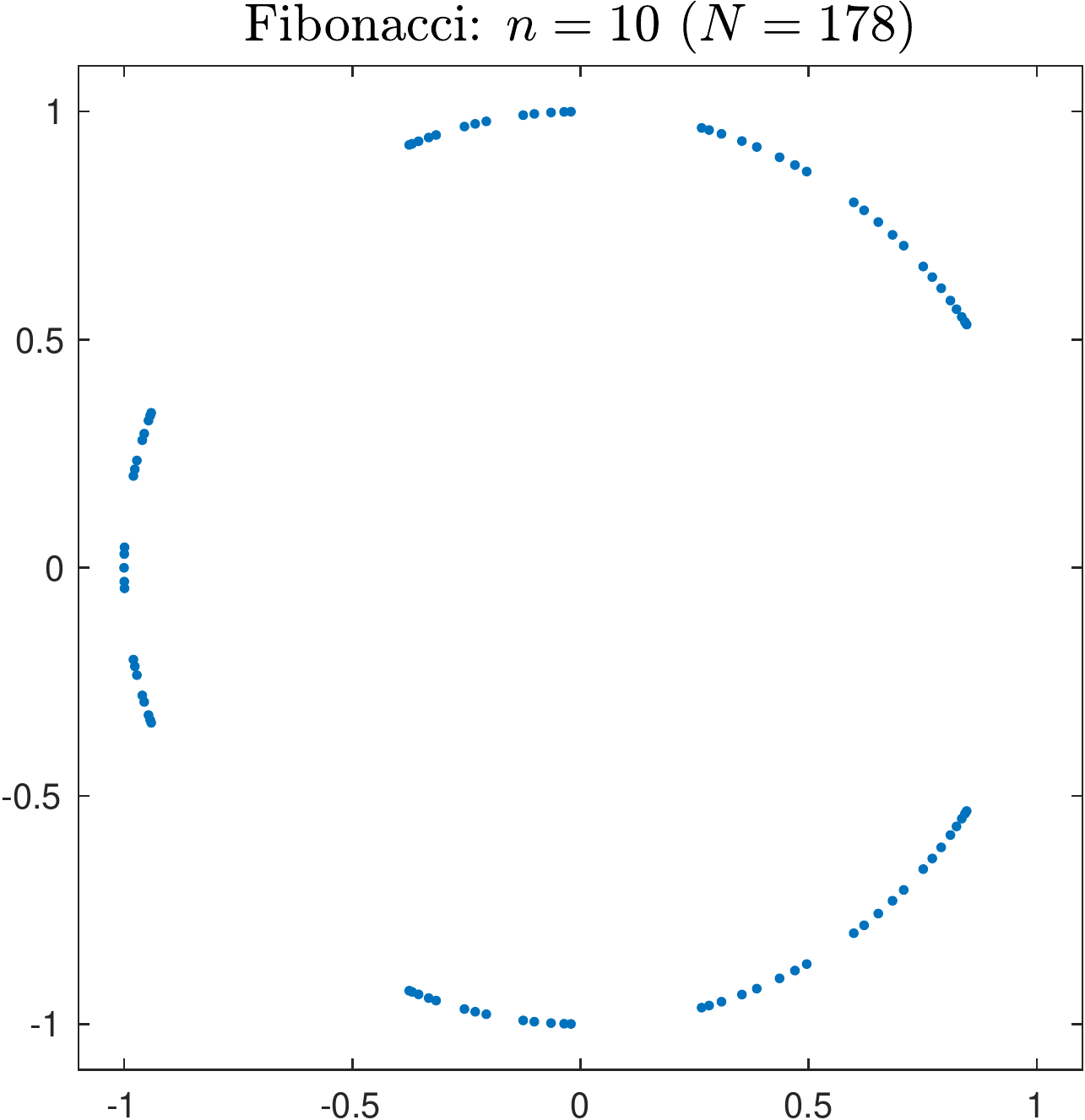}
\qquad
\includegraphics[scale=.45]{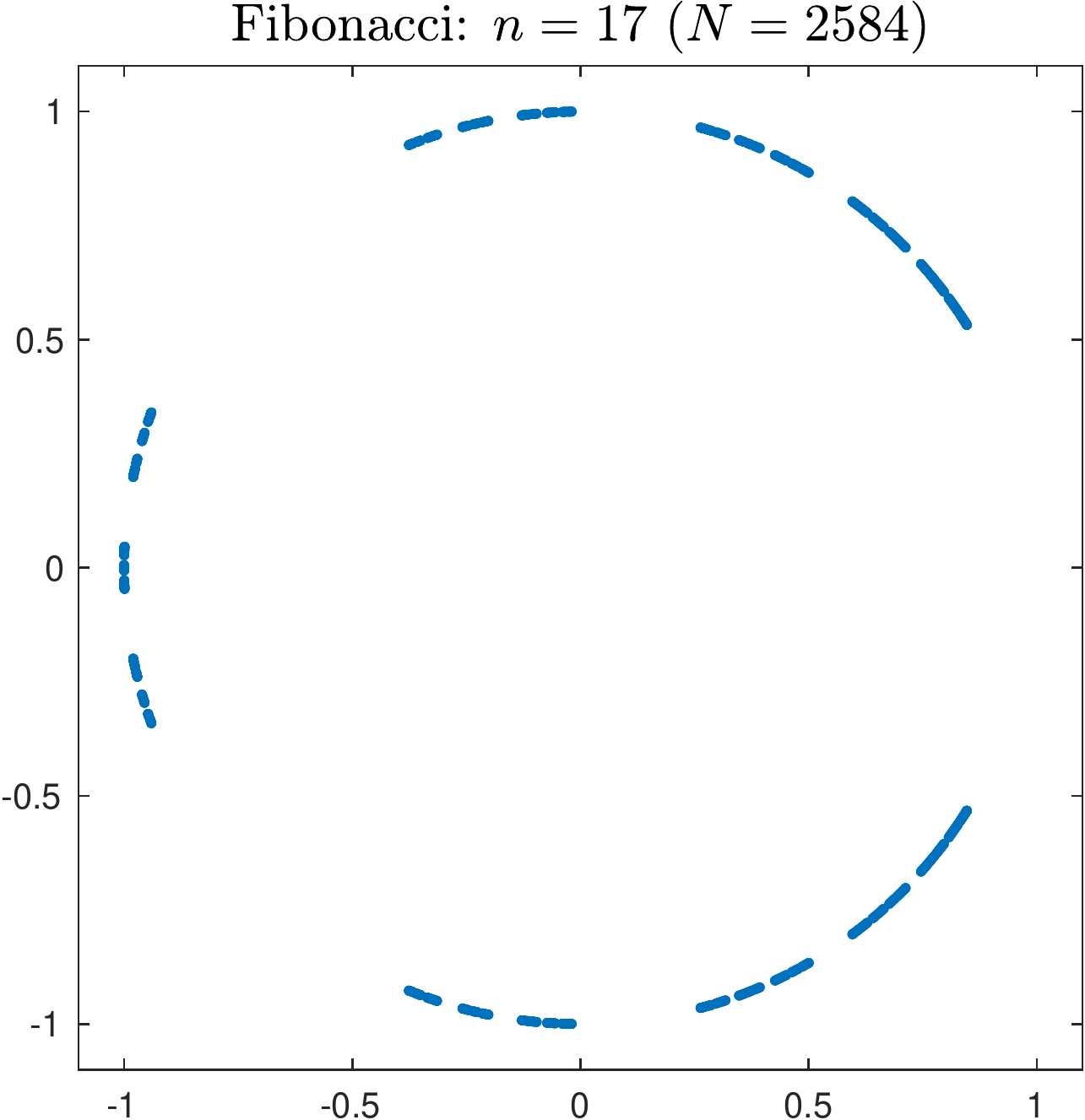}
\end{center}
Next, let us inspect the corresponding IDS, plotted as  function of $\theta = -i\log z/2\pi$, that is, we compute $\kappa([1,e^{2\pi i \theta}])$ for $\theta \in [0,1]$.
(The IDS for $n=10$ is shown in~\cite[Fig.~2]{BaaGriPis1995JSP}.) 
The fractal nature of the distribution of the zeros becomes more apparent from this perspective, and indeed can be seen to be a consequence of the Cantor structure of the spectrum of the CMV operator having coefficients generated by the Fibonacci sequence \cite{DamMunYes2013JSP}. It is known that the density of states measure assigns no weight to gaps of the spectrum, so each flat portion in the graph of the IDS corresponds to a gap in the spectrum. The height of the graph of the IDS in the gap then corresponds to the gap label.
\begin{center}
\includegraphics[scale=.45]{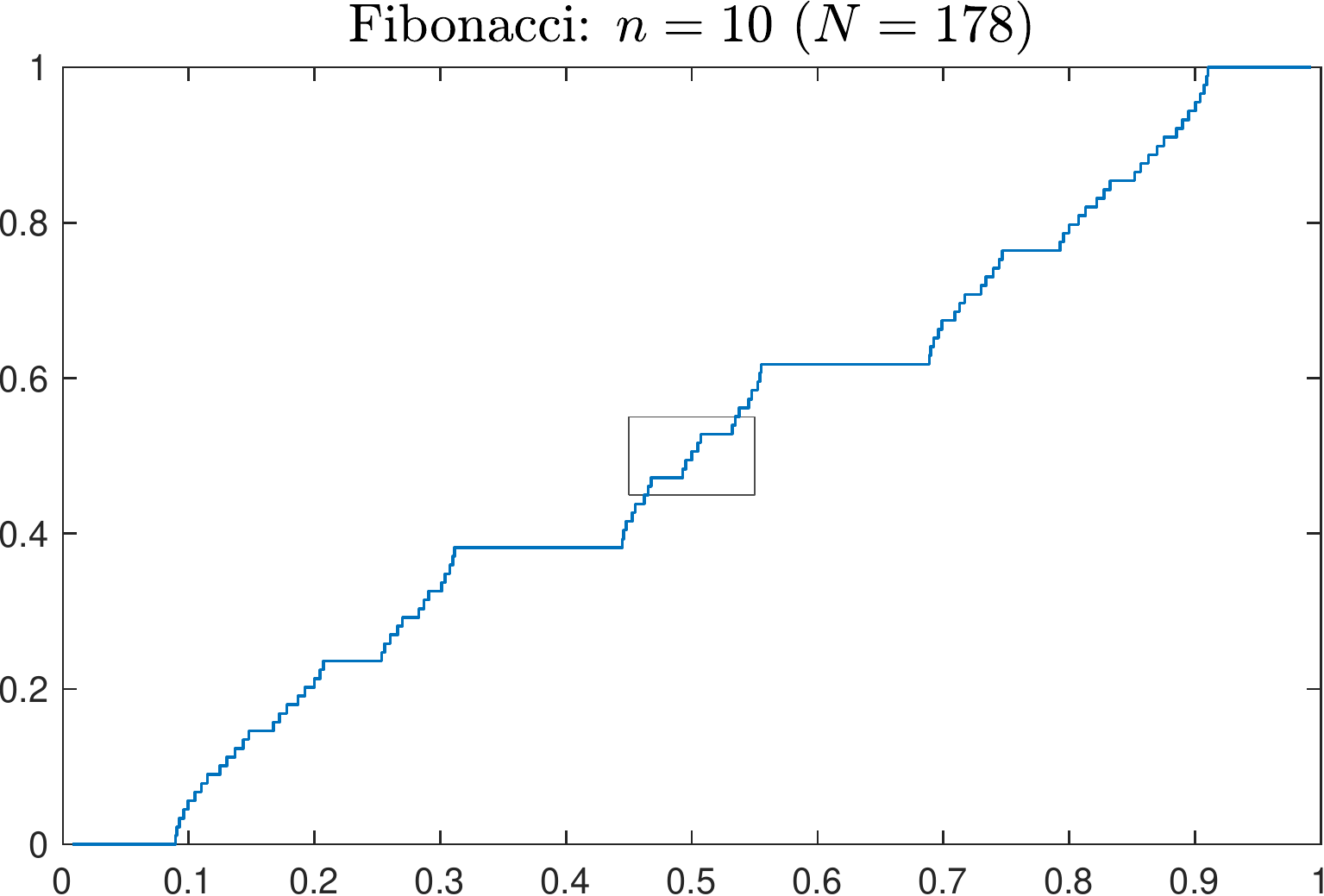}
\begin{picture}(0,0)
  \put(-184.75,72){\includegraphics[scale=0.18]{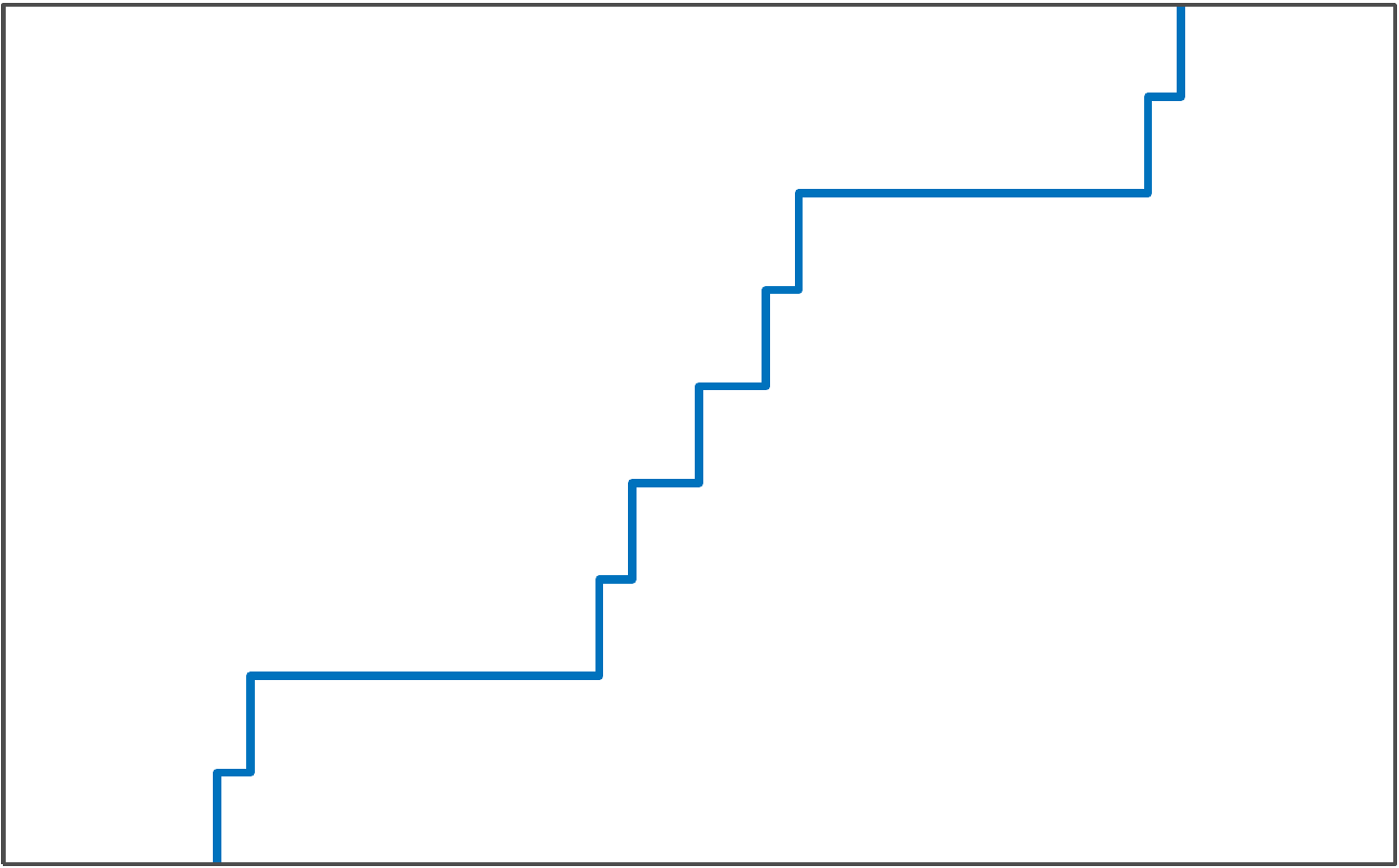}} \end{picture}
\qquad
\includegraphics[scale=.45]{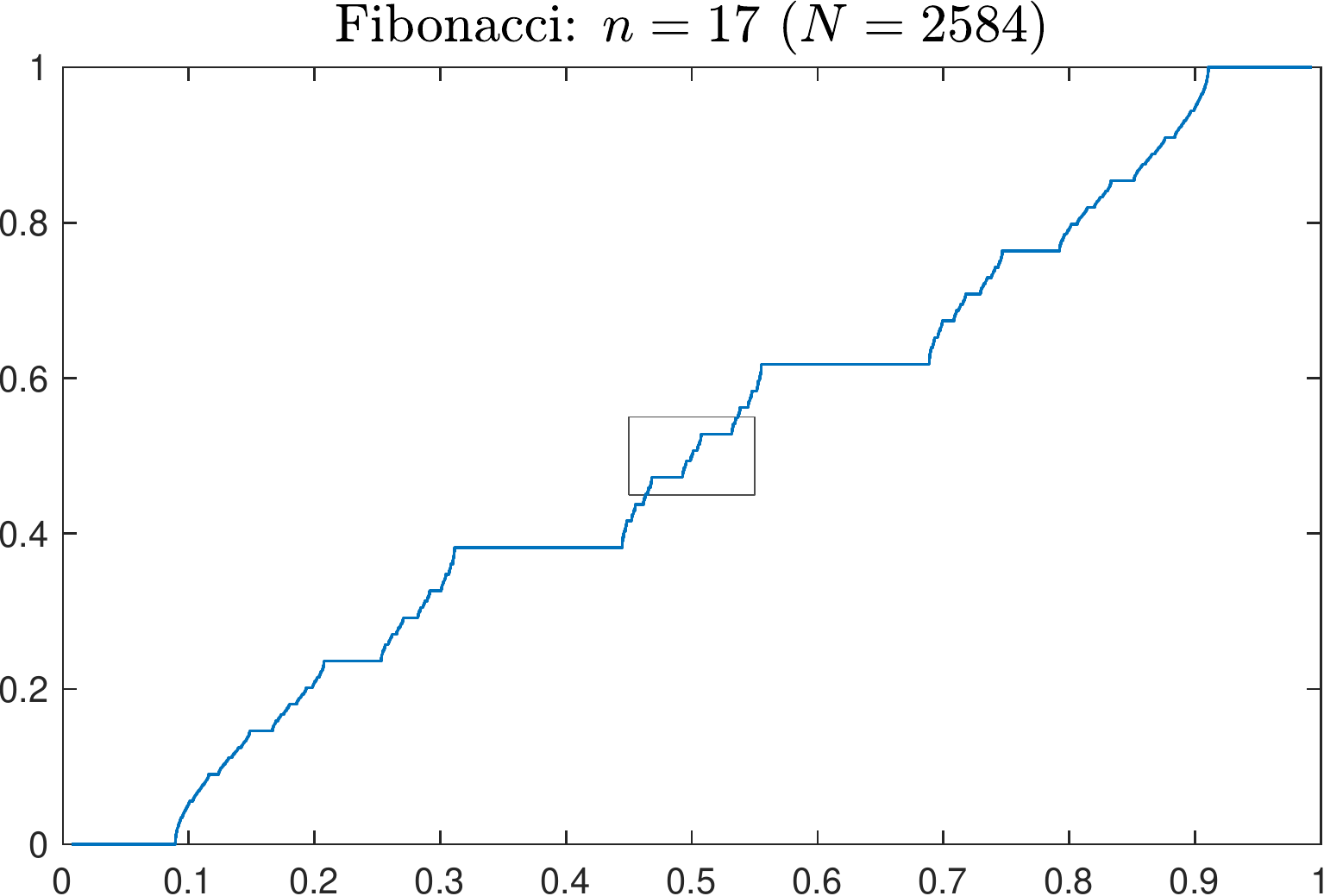}
\begin{picture}(0,0)
  \put(-184.75,71.75){\includegraphics[scale=0.18]{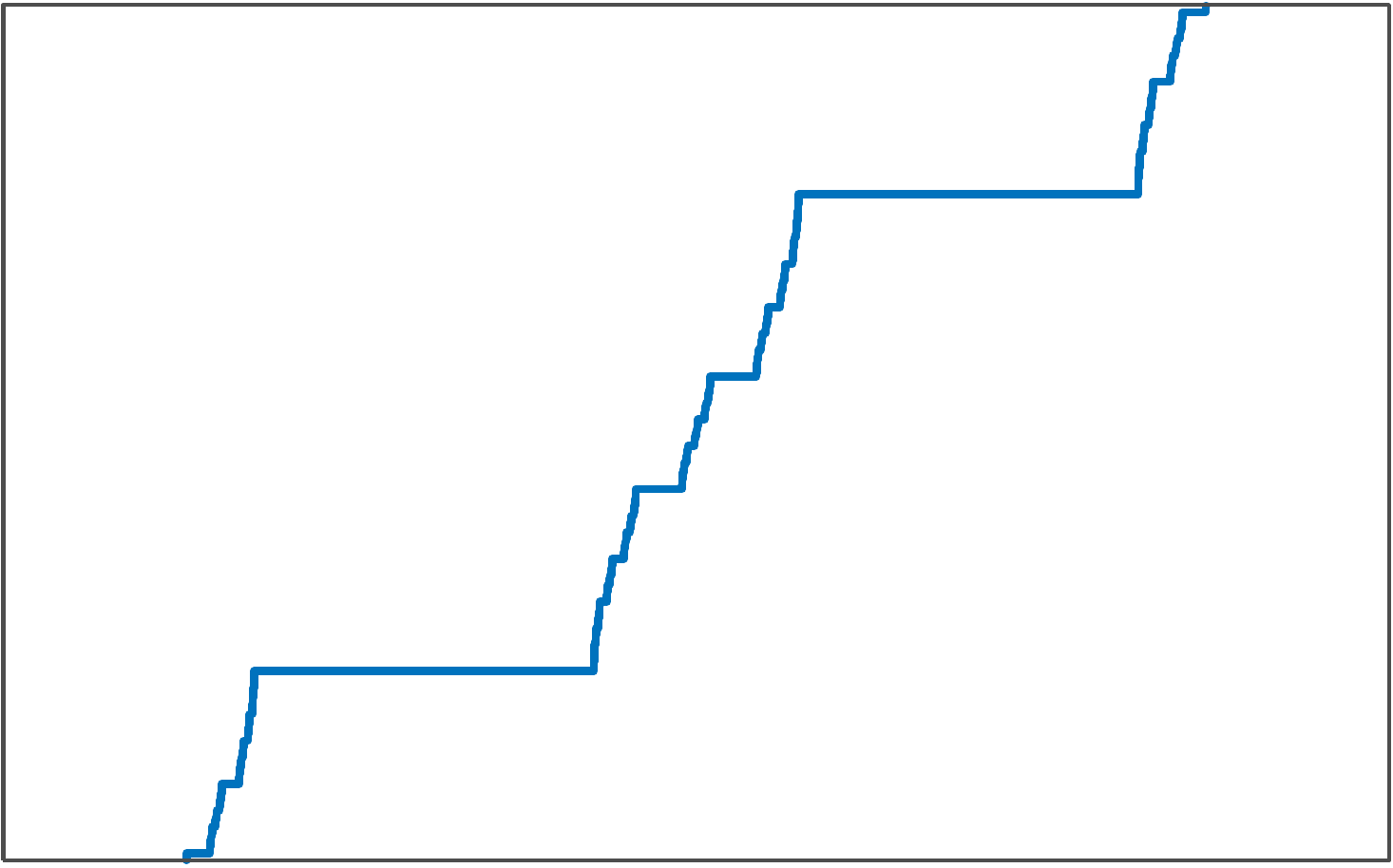}}
\end{picture}
\end{center}

To get another perspective on the structure of gaps, let us look at the distribution of gap lengths. The histograms below show the proportion of gaps between successive zeros of various lengths.
\begin{center}
\includegraphics[scale=.45]{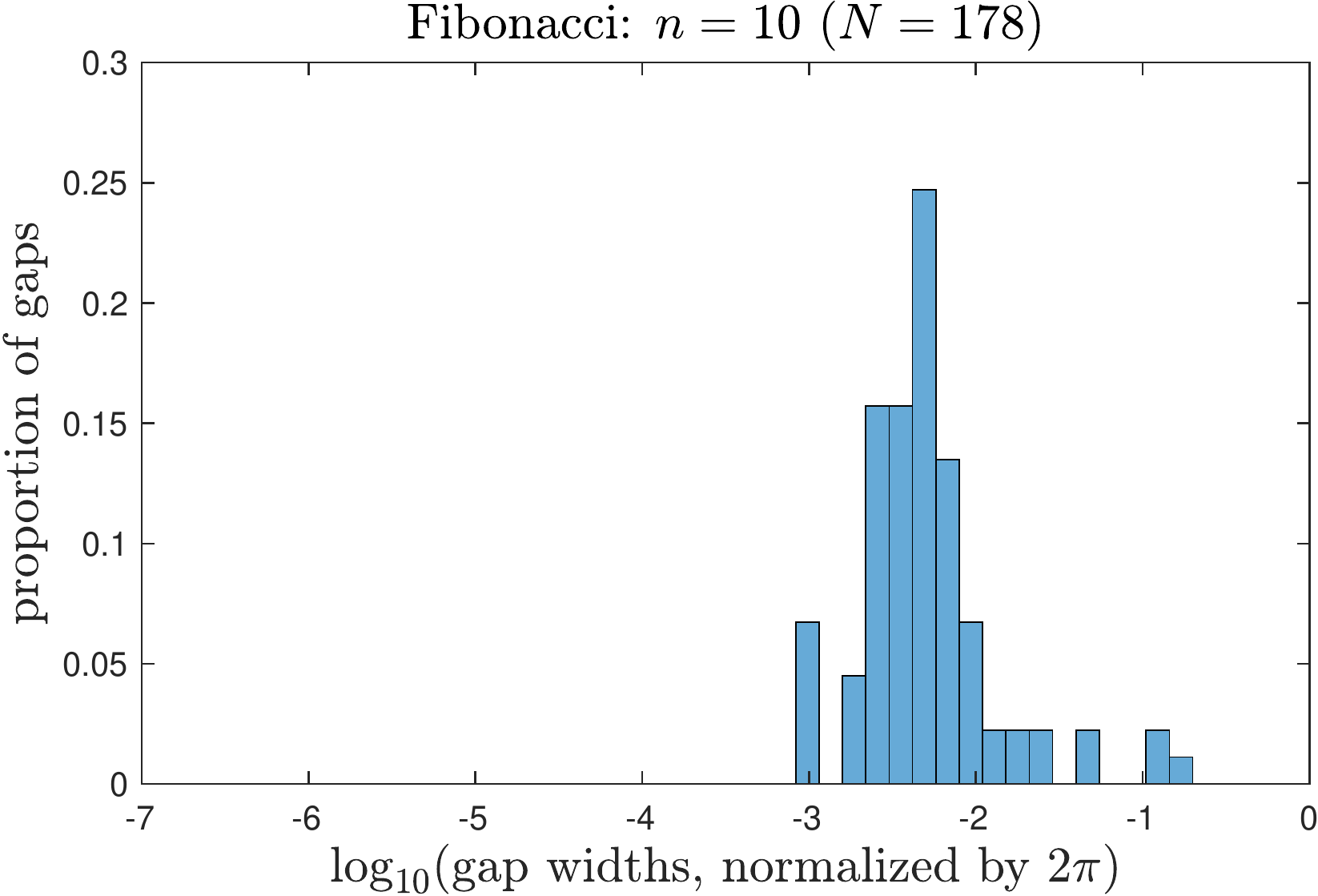}
\quad
\includegraphics[scale=.45]{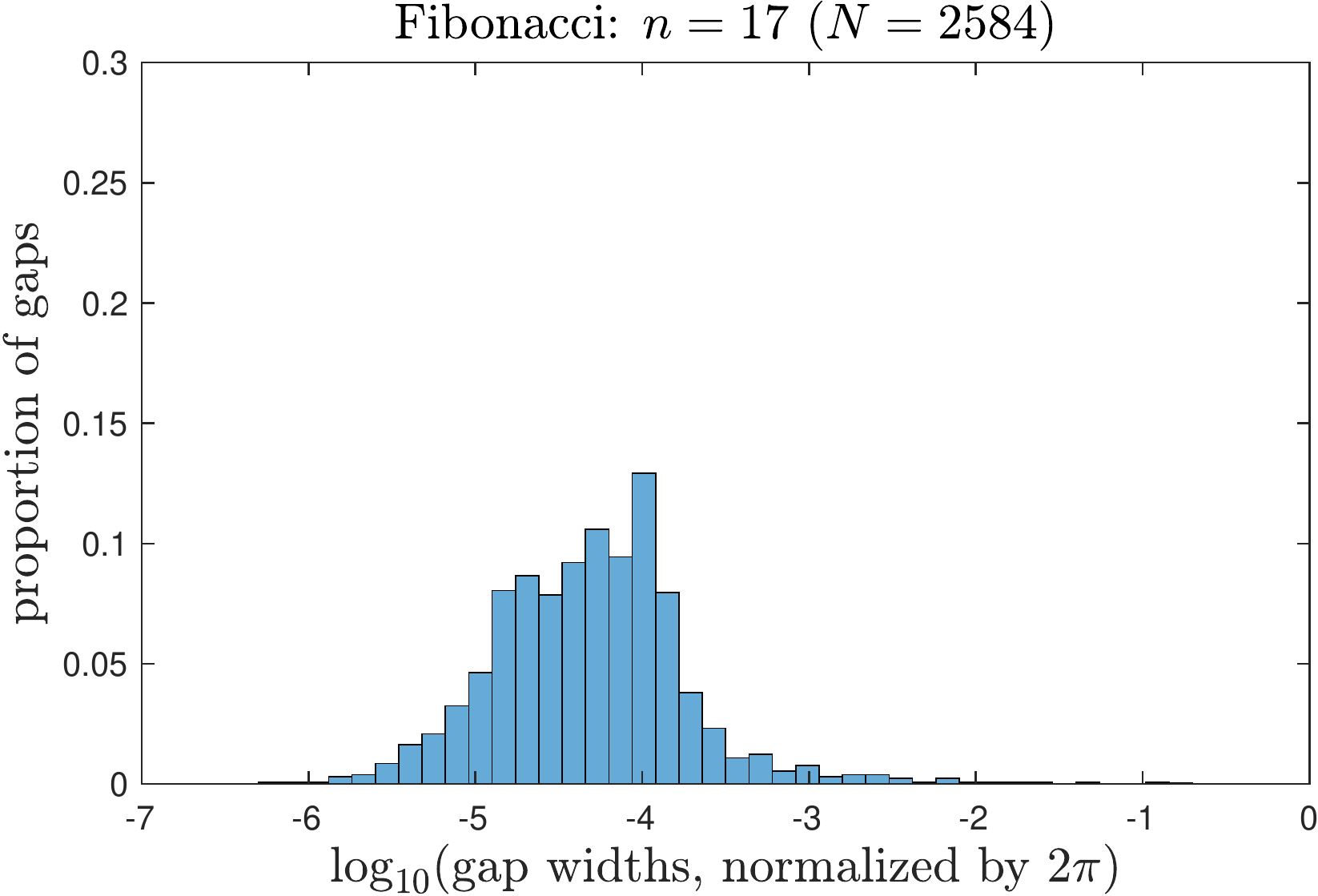}
\end{center}

\end{example}

\begin{example}[General Subshifts]
The previous example is a special case of a general type of dynamical system, called a subshift.
To formulate the general setting, consider a finite set $\calA$ (the \emph{alphabet}) with the discrete topology, and $\bbX = \calA^\bbZ$ with the product topology. This topology makes $\bbX$ a compact metrizable space with, e.g., 
\begin{equation}
d(\omega,\omega') = 2^{-\min\{|n|: \omega_n \neq \omega'_n\}}, \quad \omega \neq \omega'
\end{equation}
an example of a metric giving the topology on $\bbX$.\ \ The \emph{shift} on $\bbX$ is given by $[Tx]_n = x_{n+1}$ for $x \in \bbX$.\ \ A \emph{subshift} is any $T$-invariant compact subset of $\bbX$.\ \ If $\Omega \subseteq \bbX$ is a subshift, one abuses notation and writes $T$ for the restriction of the shift to $\Omega$.\ \ If $\mu$ is a $T$-ergodic measure on $\Omega$, it is known that
\begin{equation}
\frA(\Omega,T,\mu) = \set{\int f \, d\mu : f \in C(\Omega,\bbZ)}.
\end{equation}
Equivalently, $\frA(\Omega,T,\mu)$ can be characterized by measures of cylinder sets. More precisely, given a word $u \in \calA^*$, the associated cylinder set is
\[\Xi_u = \{\omega \in \Omega : \omega_j = u_j, \ \forall 1 \le j \le n\}.\]
Then $\frA(\Omega,T,\mu)$ is precisely the group generated by $\{\mu(\Xi_u) : u \in \calA^*\}$.

One particularly interesting class of subshifts is supplied by so-called subshifts of finite-type. Given a matrix $M \in \bbR^{\calA \times \calA}$ such that $M_{\sfa,\sfb} \in \{0,1\}$ for every $\sfa, \sfb \in\calA$, the associated \emph{subshift of finite type} is given by
\begin{equation}
\Omega_M = \{\omega \in \calA^\bbZ : M_{\omega_n,\omega_{n+1}}=1 \ \forall n \in \bbZ\}.
\end{equation}
One also assumes that $M$ is \emph{primitive} in the sense that for some $n \in \bbN$, $(M^n)_{\sfa,\sfb}>0$ for all $\sfa,\sfb \in \calA$.

There is a plethora of invariant measures on $\Omega_M$. Suppose $P \in \bbR^{\calA \times \calA}$ is such that every entry of $P$ is nonnegative, $P_{\sfa, \sfb} >0 \iff M_{\sfa,\sfb} = 1$, and 
\[\sum_{\sfb \in \calA} P_{\sfa,\sfb}=1, \quad \forall \sfa \in \calA.\]
By the primivity assumption on $M$ and the Perron--Frobenius theorem, there is a unique invariant probability vector, that is, a vector $(p_\sfa)_{\sfa \in \calA}$ such that
\begin{equation}
\sum_{\sfa \in \calA}p_\sfa P_{\sfa,\sfb} = p_\sfb, \quad \forall \sfb \in \calA.
\end{equation}
The induced measure $\mu$ on $\Omega_M$ is given by
\begin{equation} \label{eq:muXiu}
\mu(\Xi_u) = p_{u_1}\prod_{j=1}^{n-1}P_{u_j,u_{j+1}}.
\end{equation}
In view of the previous discussion, one can use this to compute $\frA(\Omega,T,\mu)$ in terms of the entries of $P$ and $p$, namely, $\frA(\Omega,T,\mu)$ is the $\bbZ$-module generated by the numbers in \eqref{eq:muXiu}.
\end{example}

\begin{example}[Cat Map]
We begin with the base space $\Omega = \bbT^2 := \bbR^2/\bbZ^2$. The \emph{cat map} is the  transformation $T = T_{\rm cat}:\Omega \to \Omega$ given by
\begin{equation}
T_{\rm cat}(x,y) = (2x+y,x+y), \quad (x,y) \in \bbT^2.
\end{equation}
This example is known to have many invariant measures. One can check that $\mu = \mathrm{Leb}$, the normalized Lebesgue measure on $\bbT^2$, is $T_{\rm cat}$-ergodic. It was shown in \cite{DFGap} that
\begin{equation}
\frA(\Omega,T_{\rm cat}, \mu) = \bbZ.
\end{equation}

For the sampling function, we take $f(x,y) = 1/2 + \cos(2\pi y)/3$. The corresponding Verblunsky coefficients are
\begin{equation}
\alpha_n(x,y)= f(T^n(x,y)), \quad (x,y) \in \bbT^2.
\end{equation}
In view of the relationship we have discussed previously, this corresponds to an Ising model with couplings
\[p_n = p_n(x,y) = - \frac{1}{2}\log f(T^n(x,y)). \]

By induction, one can check that
\[T_{\rm cat}^n(x,y) = (F_{2n+1}x+F_{2n}y, F_{2n}x+F_{2n-1}y),\]
where $F_n$ denotes the $n$th Fibonacci number, normalized by $F_0=0$, $F_1=1$, and $F_{n+1}= F_n + F_{n-1}$ for $n \geq 1$. 

For the cat map, we produce plots similar to those from previous examples. 
We note that the numerical calculations in the illustrations that follow require
more care than might first be apparent.  The rapid growth of the entries in $T_{\rm cat}^n(x,y)$ 
means that, in standard double-precision floating point arithmetic~\cite{Ove01}, 
the argument $y$ in $\cos(2\pi y)$ in the calculation of $\alpha_n(x,y)$ lacks
sufficient precision for the Verblunsky coefficients to be computed accurately.
Indeed, double precision calculations produce $\alpha_n(x,y)$ with errors of $O(1)$ when $n\ge 40$, 
rendering subsequent numerically computed eigenvalue statistics essentially meaningless.
To avoid this pitfall, we compute these coefficients using high-precision arithmetic in Mathematica,
then render them in full double-precision accuracy for the subsequent eigenvalue calculation.
(We use MATLAB's standard dense nonsymmetric eigensolver {\tt eig} to compute the 
eigenvalues of the unitary matrix $\calF_N(\pi/2)$.)

We start with the zeros, generated with $x=1/\sqrt{2}$ and $y = 1/\sqrt{3}$.
In the plots below, the red arc $[e^{-i \phi}, e^{i \phi}]$ 
for $\phi = 2 \sin^{-1}(1/6)$ denotes an inner bound on the spectral gap proved in \cite{DFLY2015IMRN}.

\begin{center}
\includegraphics[scale=.45]{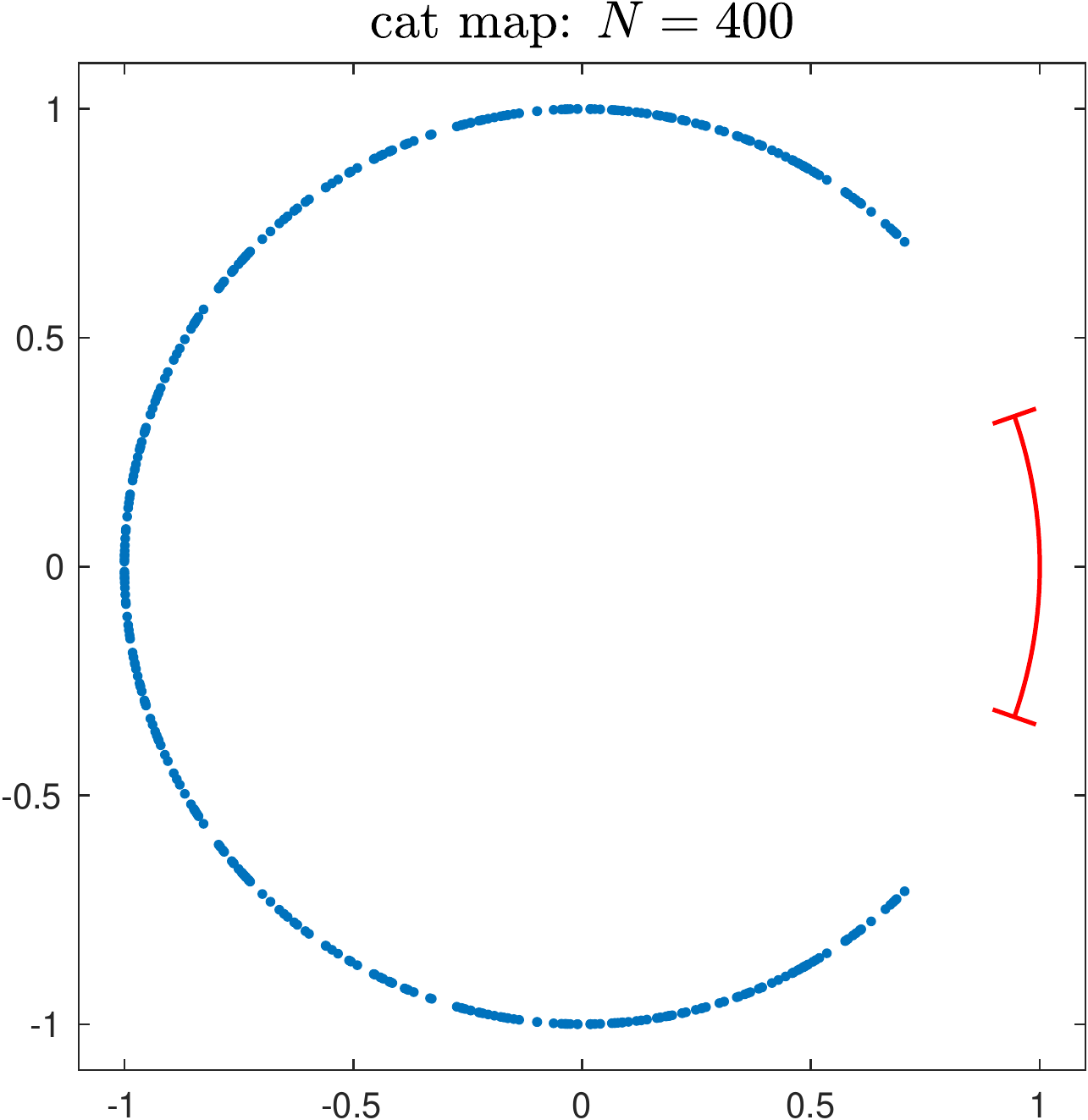}
\quad
\includegraphics[scale=.45]{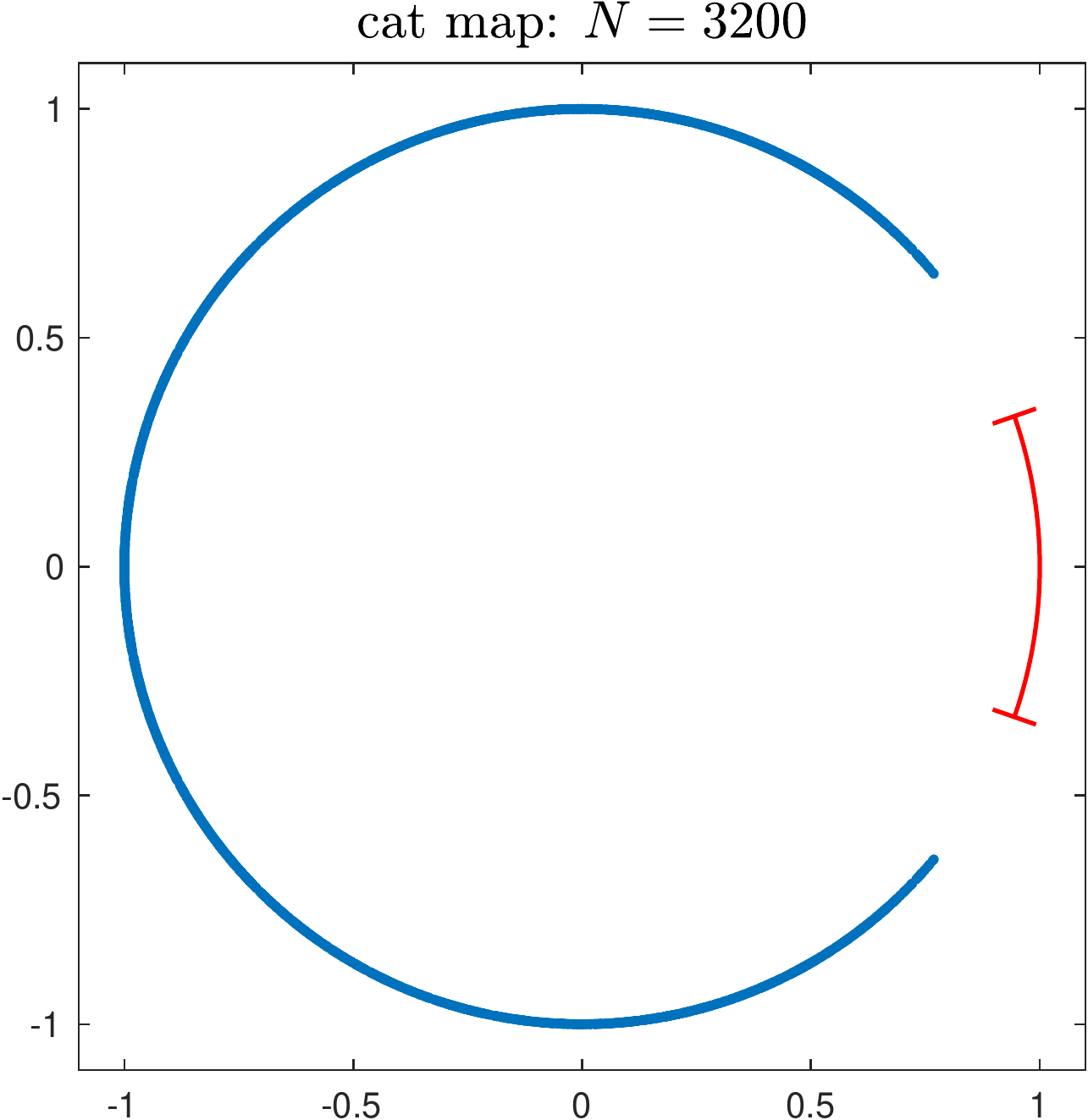}
\end{center}

As the reader can see, the zeros densely fill the circle, which is expected given results obtained rigorously from the gap labelling theorem. Below, we show the IDS as a function of $\theta = -i \log z / 2\pi$ as before.
\begin{center}
\includegraphics[scale=.45]{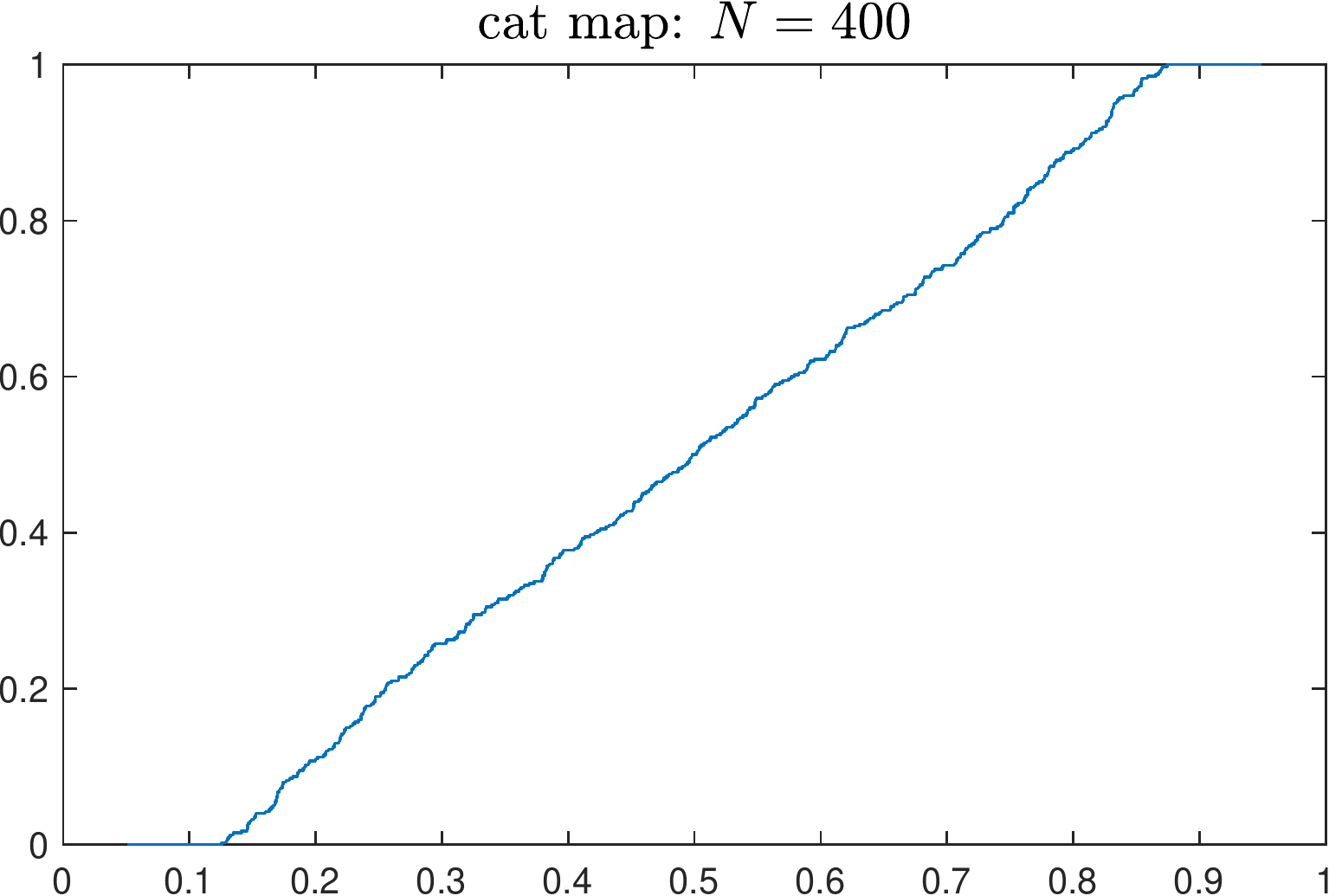}
\quad
\includegraphics[scale=.45]{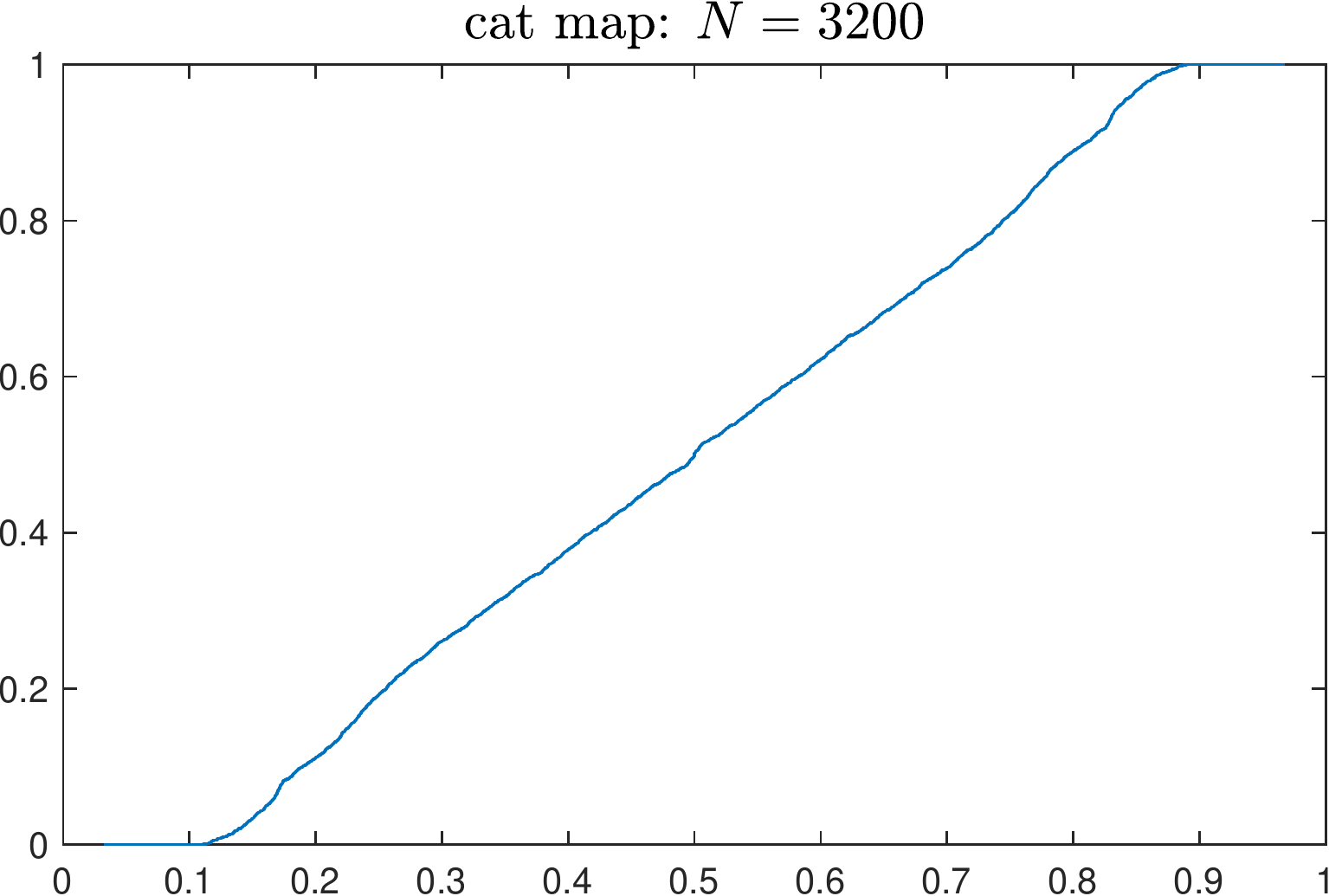}
\end{center}
As one can see, the graph comports with the general picture of an absence of gaps, since it appears (as it must) that the IDS is everywhere increasing on a suitable arc.

Finally, we conclude with the distribution of gap lengths. 
\begin{center}
\includegraphics[scale=.45]{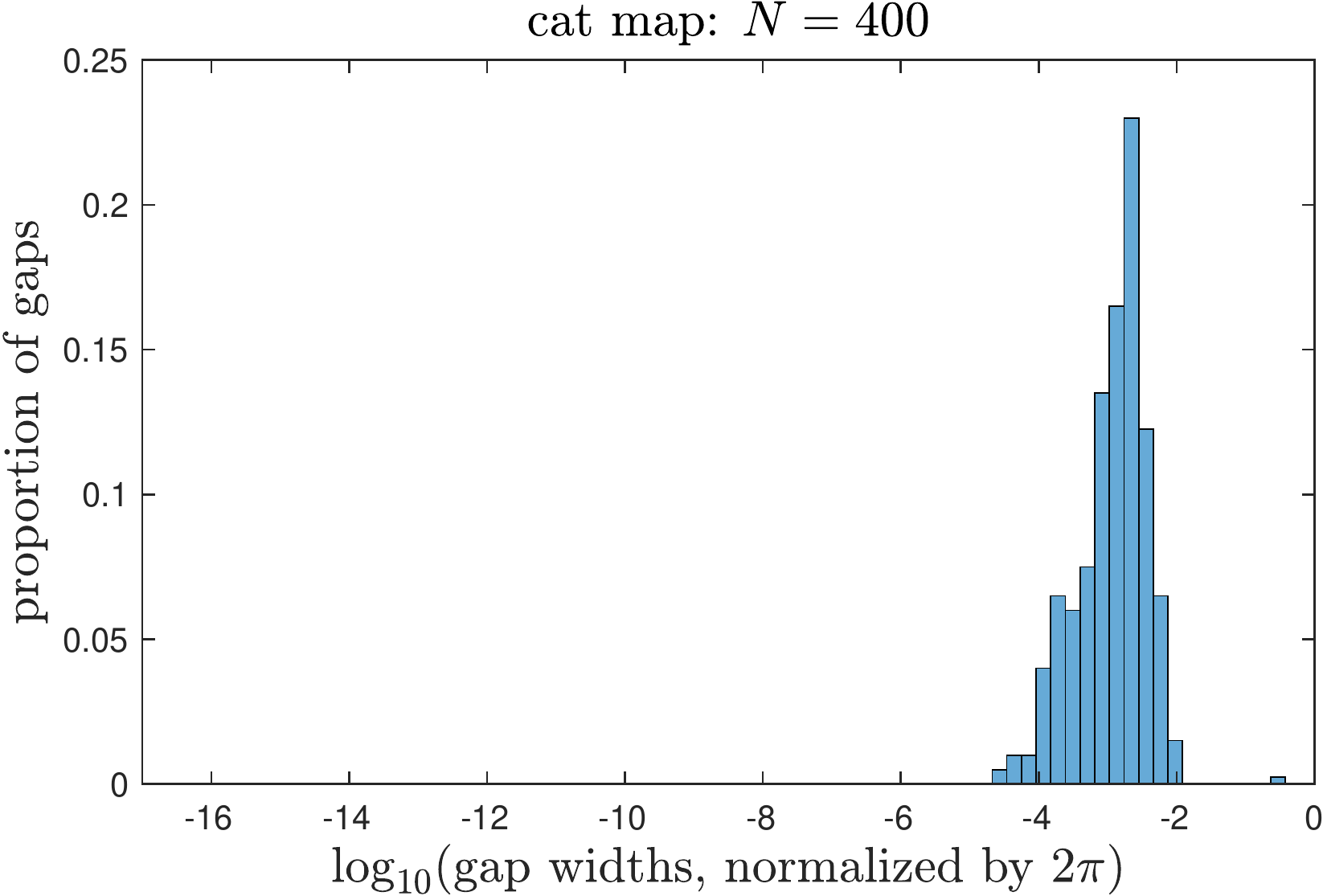}
\quad
\includegraphics[scale=.45]{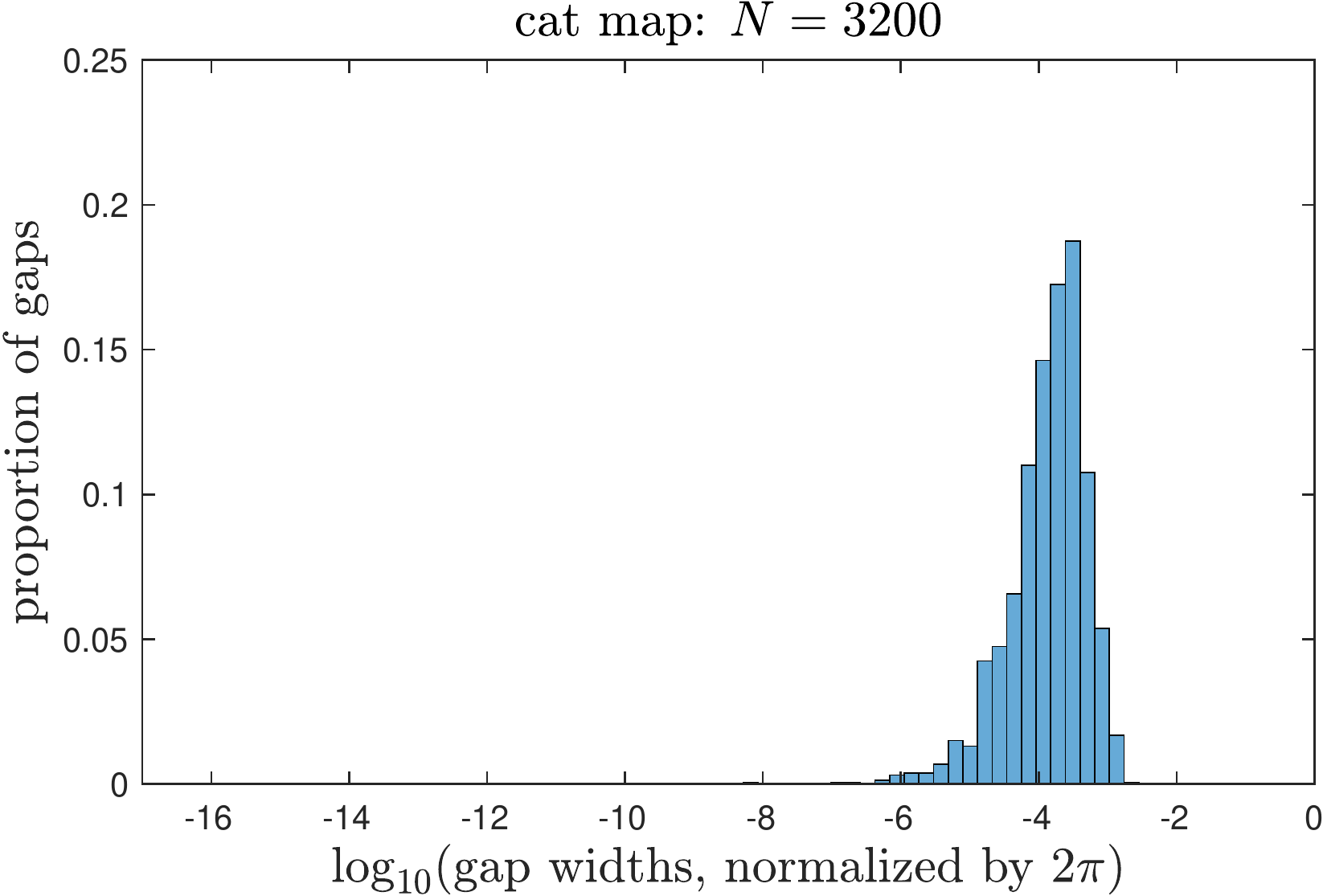}
\end{center}
Here, one observes something rather curious: the lengths of the gaps seem to be more uniform than one might expect. Concretely, from existing work on Schr\"odinger operators \cite{BourgainSchlag2000CMP}, one would expect CMV matrices with coefficients generated by the cat map to exhibit Anderson localization, that is, pure point spectrum with exponentially decaying eigenfunctions. Then, based on the same reasoning that one pursues for Schr\"odinger operators, one would expect the distribution of eigenvalues to exhibit less repulsion.

Let us note that the cat map has a natural Markov partition and hence is a factor of a subshift of finite type in a natural way. Concretely, taking $\calA = \{1,2,3,4,5\}$ and 
\begin{equation}
M_{\rm cat} = \begin{bmatrix} 1 & 0 & 1 & 1 & 0 \\  1 & 0 & 1 & 1 & 0 \\ 1 & 0 & 1 & 1 & 0 \\ 0 & 1 & 0 & 0 & 1 \\ 0 & 1 & 0 & 0 & 1. \end{bmatrix},
\end{equation}
there is a continuous factor map $\Phi:\Omega_M \to \bbT^2$ such that $\Phi \circ T = T_{\rm cat}\circ \Phi$, where $T$ denotes the shift on $\Omega_M$. 
See Brin--Stuck~\cite[pp.~135--137]{BrinStuck2015Book} or Katok--Hasselblatt~\cite[Section~20]{KatokHassel1995Book} for more discussion and details.

What is notable about this example is the disparity between the label sets. For the subshift $\Omega_M$ with invariant measure $\mu$, the label set is a dense subgroup of $\bbR$. However, the label set for the cat map itself is $\bbZ$, which leads to the conclusion that the almost-sure spectrum associated with a cat map model is connected. 
\end{example}

\begin{example}[Skew Shift]
The base space is $\bbT^2 = \bbR^2/\bbZ^2$ and the transformation $T = T_{\rm ss}$ is given by
\begin{equation}
T_{\rm ss}(x,y) = (x+\gamma,x+y)
\end{equation}
where $\gamma \in \bbR$ is a fixed irrational number. For this example, the label set can be shown to be
\begin{equation} \bbZ + \gamma \bbZ = \{n+m\gamma: n,m \in \bbZ\}, \end{equation}
which is a dense subgroup of $\bbR$.

For the sampling function, we take $f(x,y) = 1/2+ \cos(2\pi y)/3$ as before. The Verblunsky coefficients are
\begin{equation}
\alpha_n(x,y)= f(T^n(x,y)), \quad (x,y) \in \bbT^2.
\end{equation}
As before, one can use induction to write $T_{\rm ss}^n$ explicitly for $n \in \bbN$ as
\begin{equation}
T_{\rm ss}^n(x,y)= \left(x+n\gamma, y+nx+\frac{n(n-1)}{2}\gamma\right).
\end{equation}
Taking $(x,y) = (\gamma/2,0)$ for the starting point leads to
\begin{equation}
\alpha_n(\gamma/2,0) = 1/2+ \cos(n^2 \pi  \gamma)/3, \quad n \in \bbN.
\end{equation}

As in the case of the cat map, we show the zeros for truncations of order $N=400$ and $N=3200$ for the case $\gamma = 1/\sqrt{2}$, together with a red arc showing an inner bound on the main spectral gap about $z=1$ (indeed the same inner bound as in the cat map example). As before, we also show the integrated density of states and a histogram showing the distribution of gap lengths.

\begin{center}
\includegraphics[scale=.45]{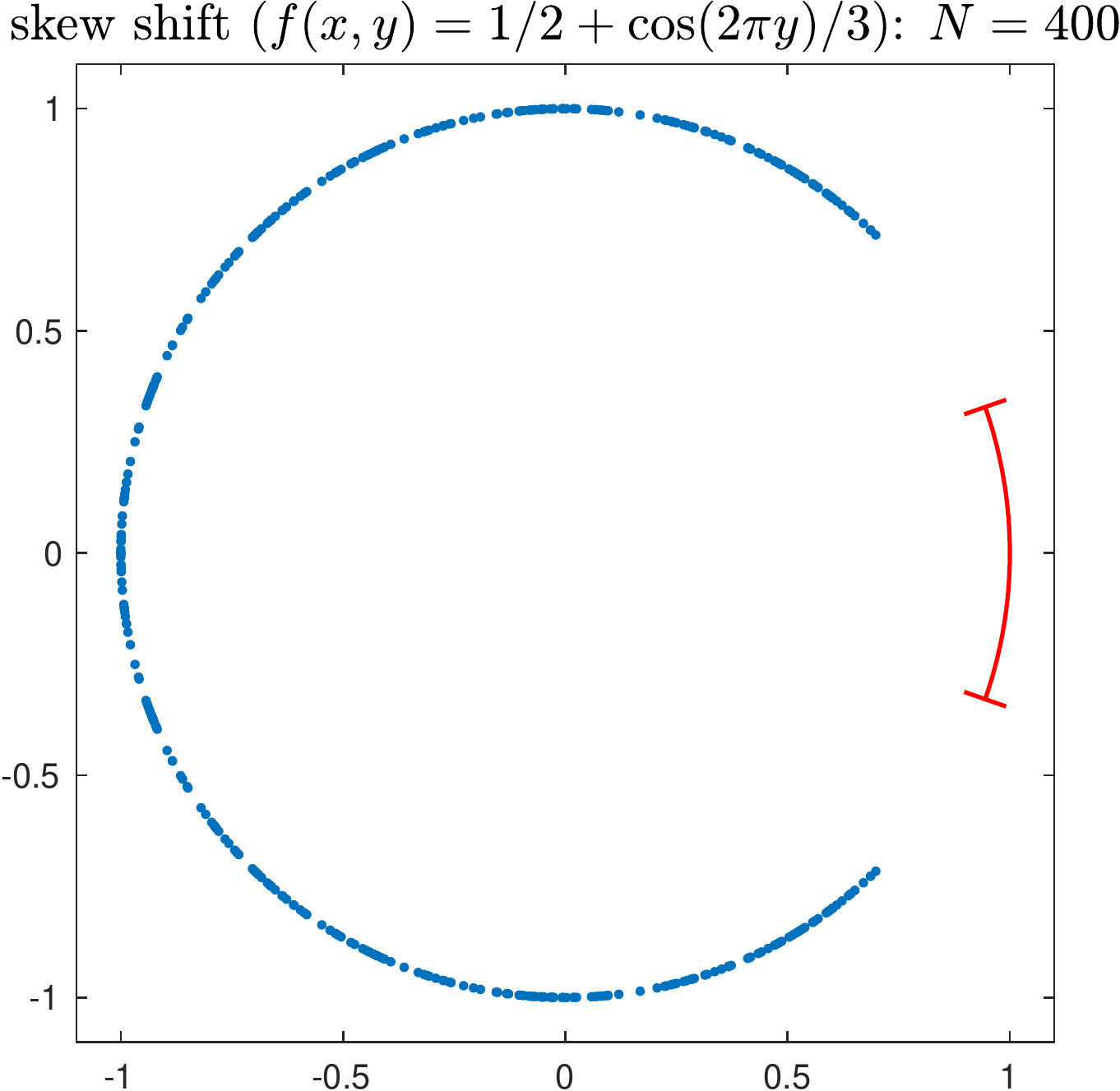}
\quad
\includegraphics[scale=.45]{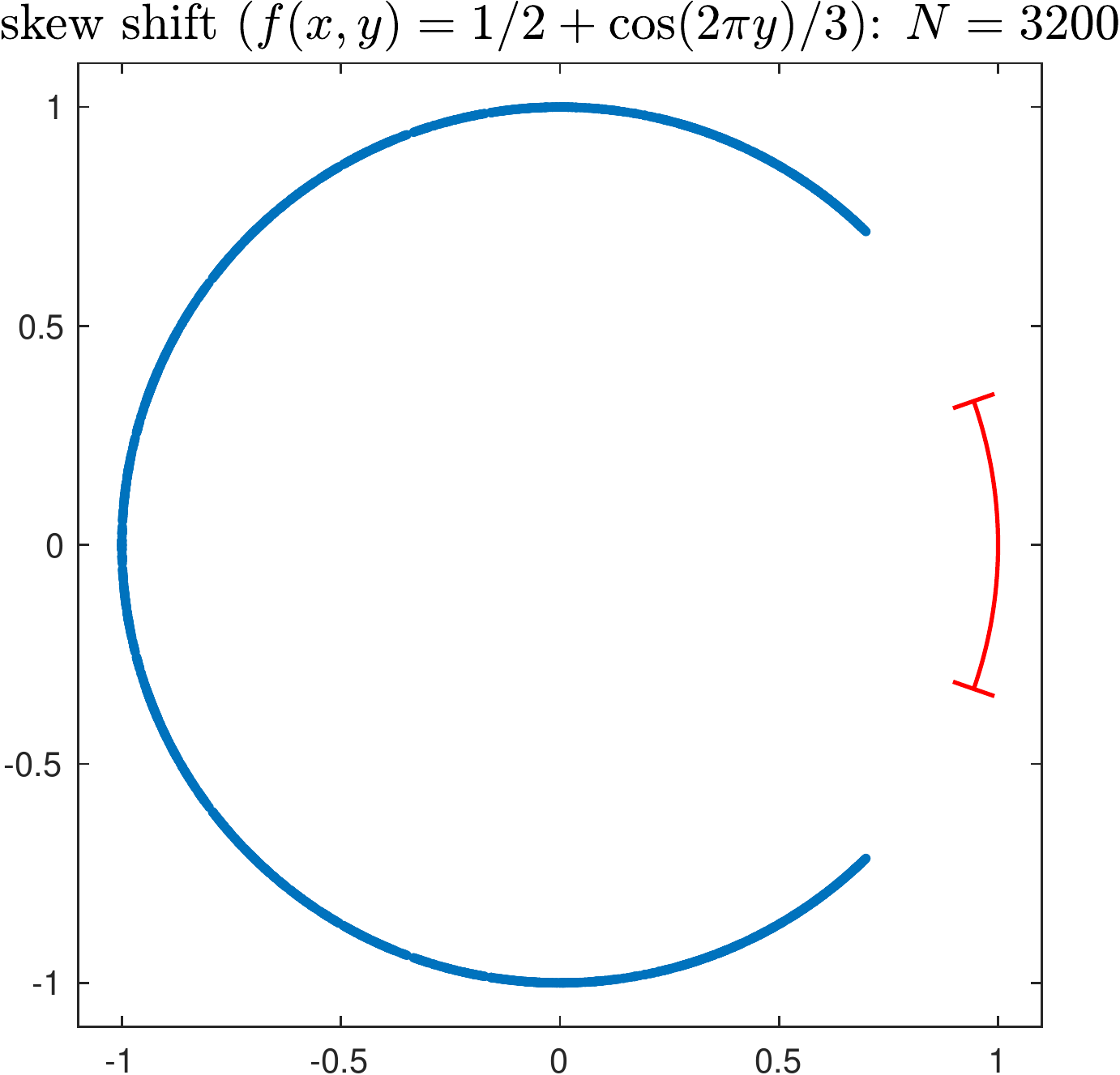}
\end{center}

\begin{center}
\includegraphics[scale=.45]{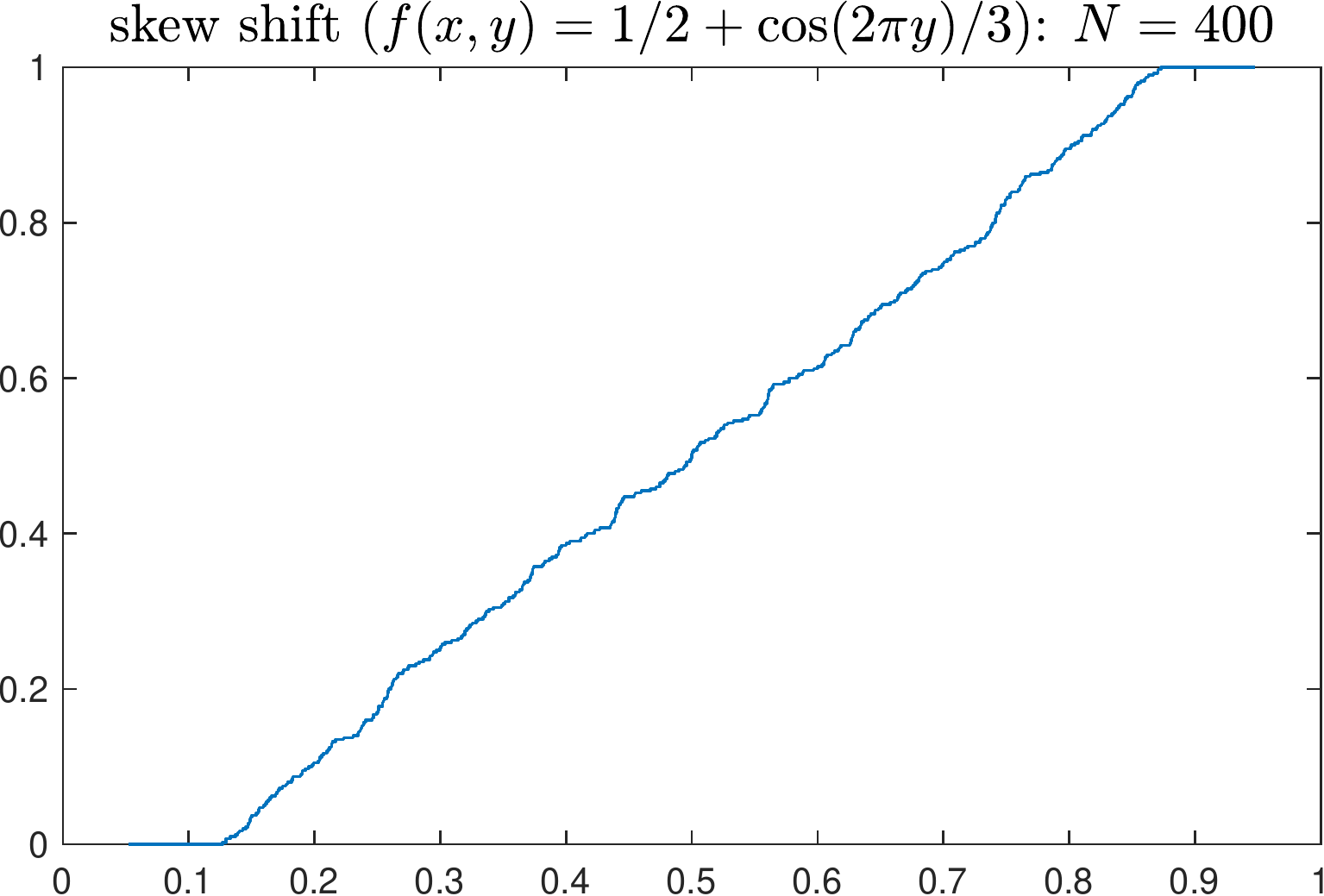}
\quad
\includegraphics[scale=.45]{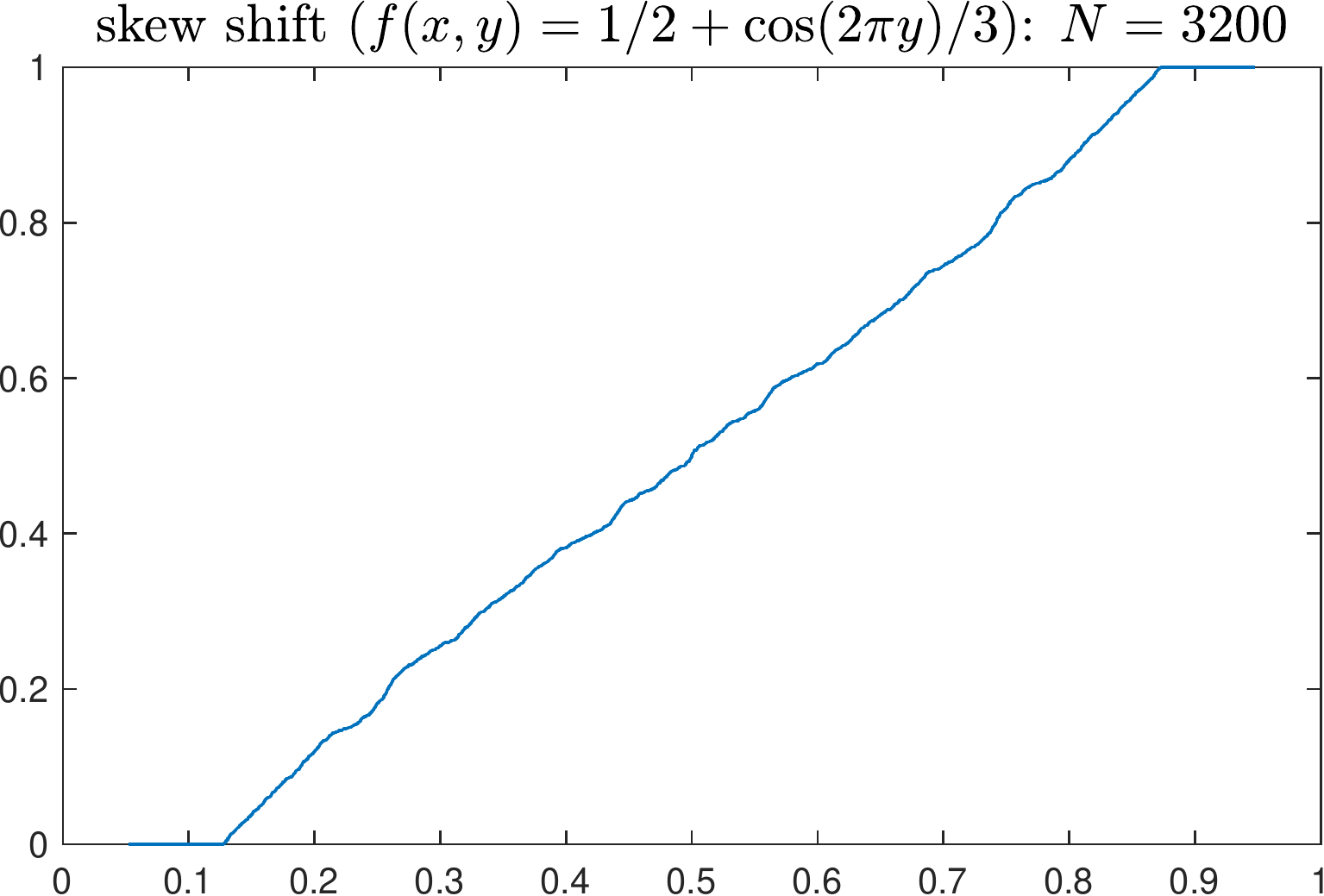}
\end{center}

\begin{center}
\includegraphics[scale=.45]{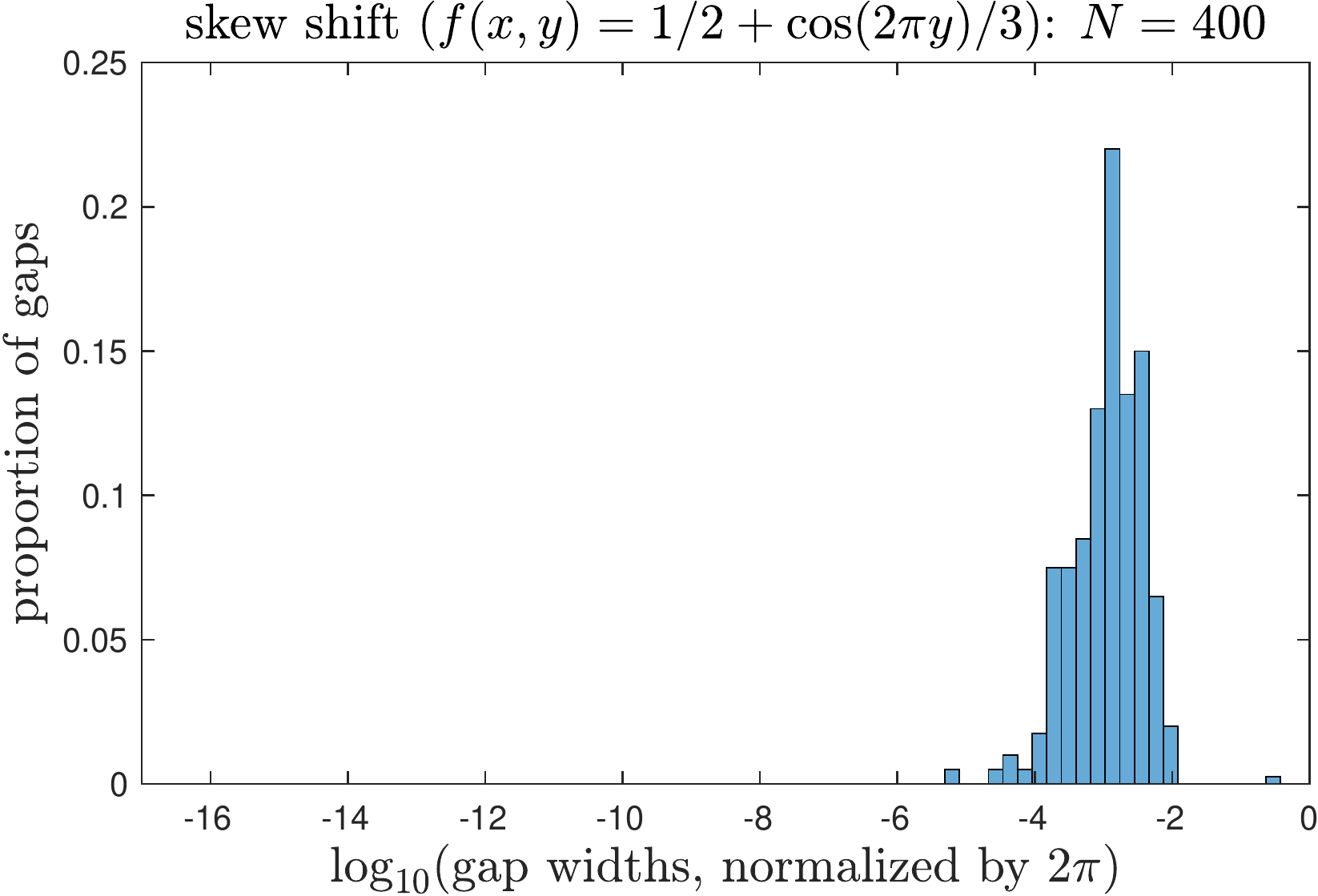}
\quad
\includegraphics[scale=.45]{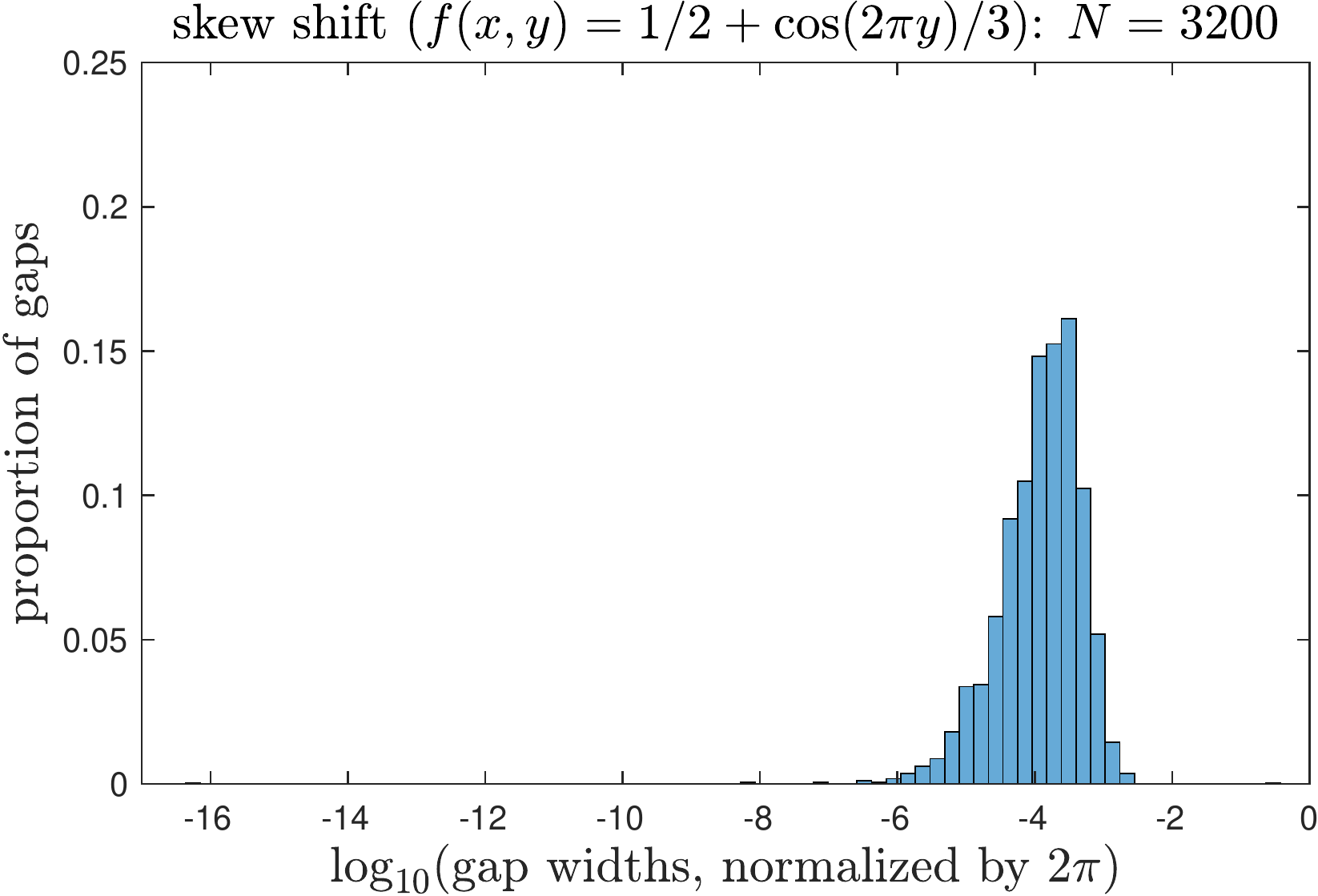}
\end{center}

While the application to Ising necessitates choosing strictly positive Verblunsky coefficients,  it is also of interest to look at sign-indefinite models from the CMV perspective. A notably interesting example is given by using the sampling function $f(x,y) = \lambda \cos(2\pi y)$ for some $0<\lambda<1$. The corresponding figures appear below. 
(Both skew shift examples give a numerically computed gap of width zero for $N=3200$, not shown on the histograms.)

\begin{center}
\includegraphics[scale=.45]{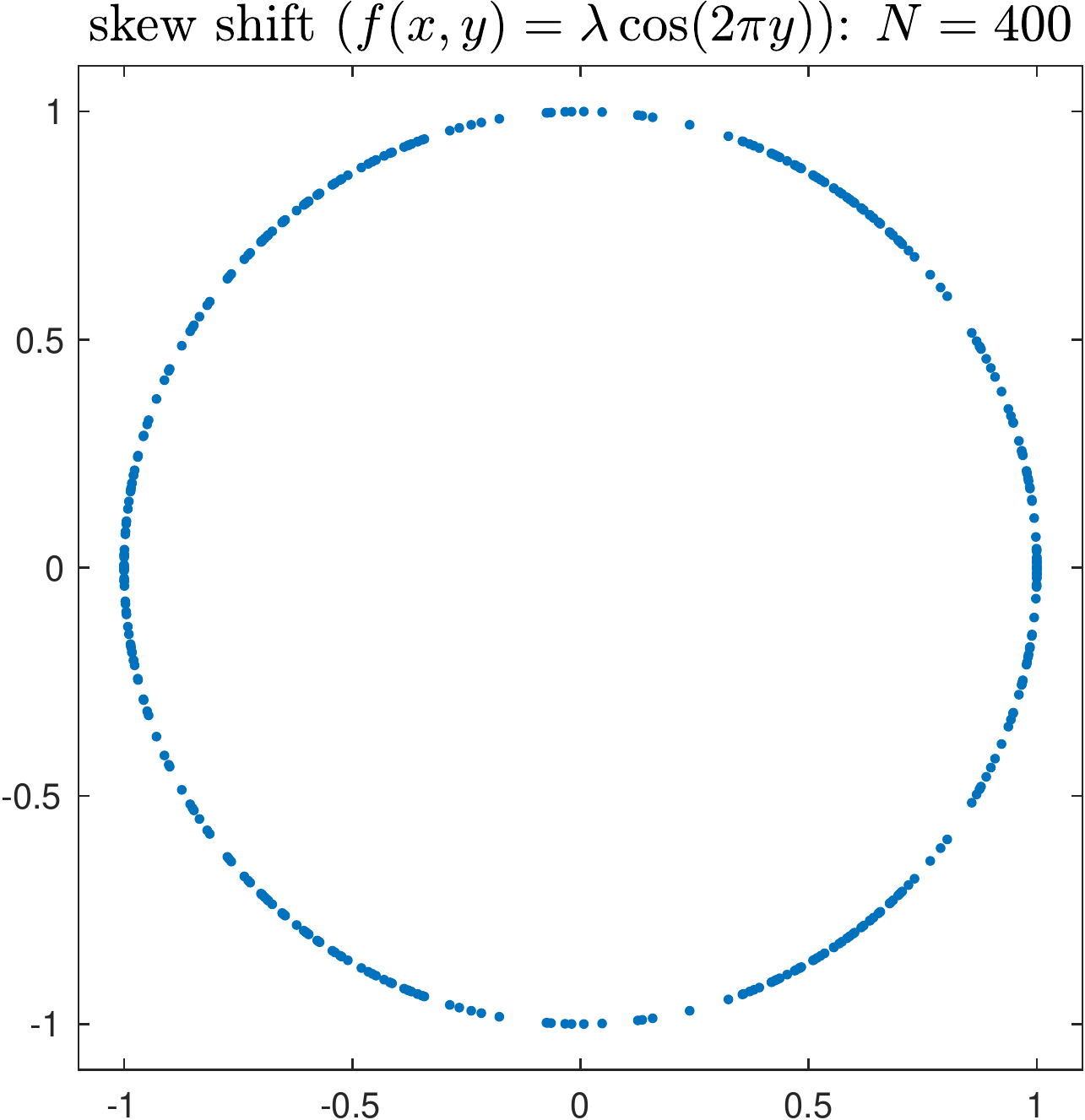}
\quad
\includegraphics[scale=.45]{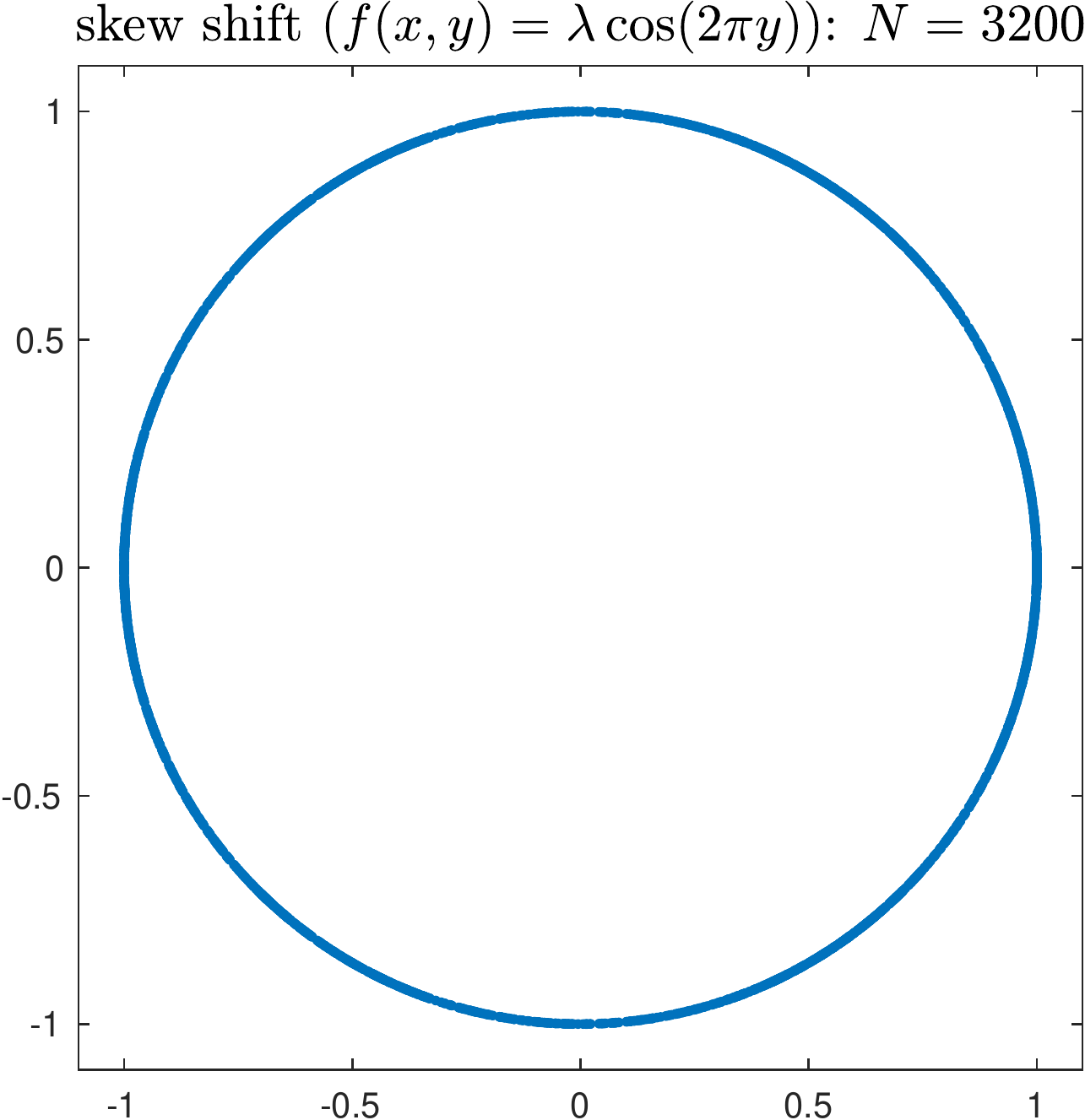}
\end{center}

\begin{center}
\includegraphics[scale=.45]{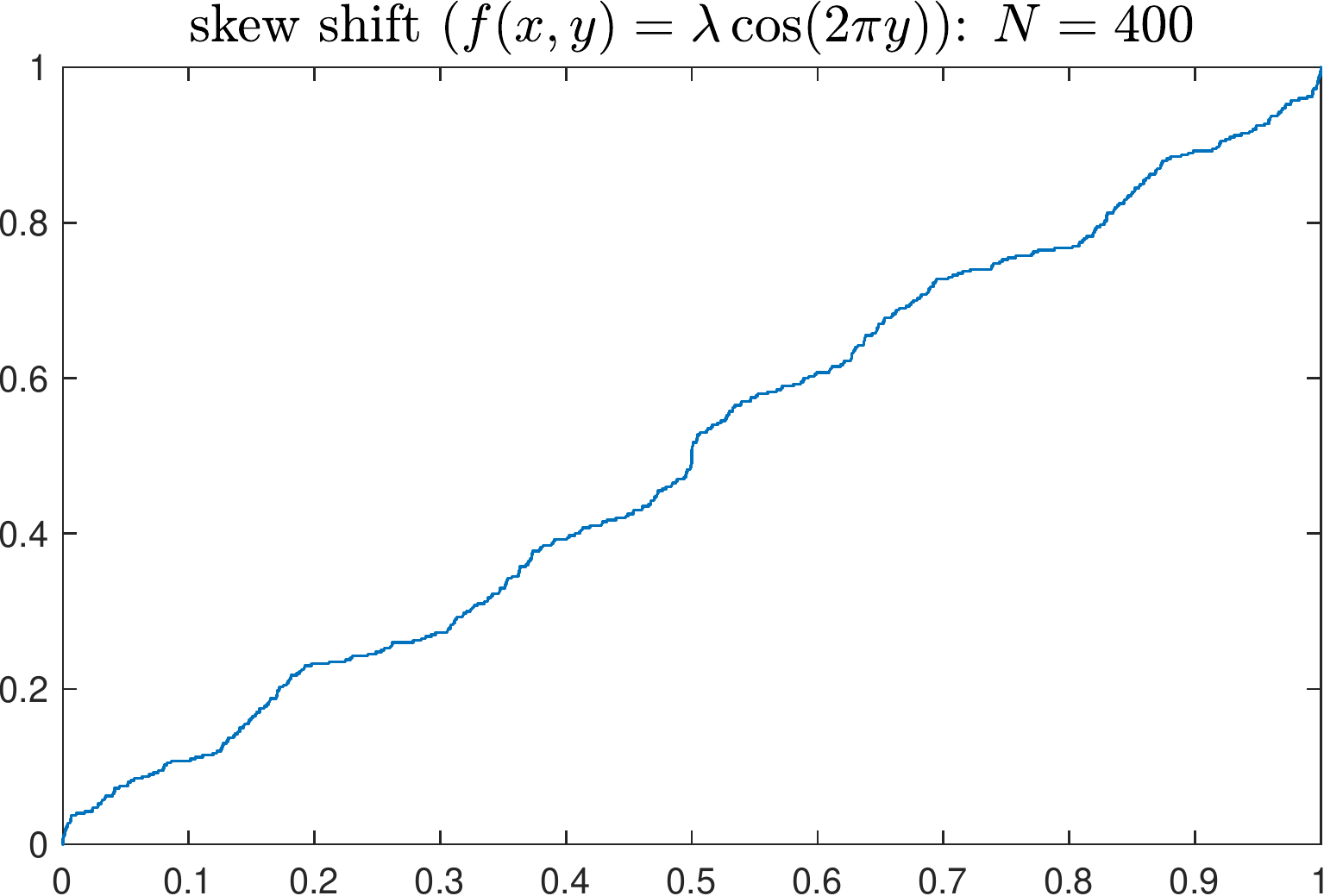}
\quad
\includegraphics[scale=.45]{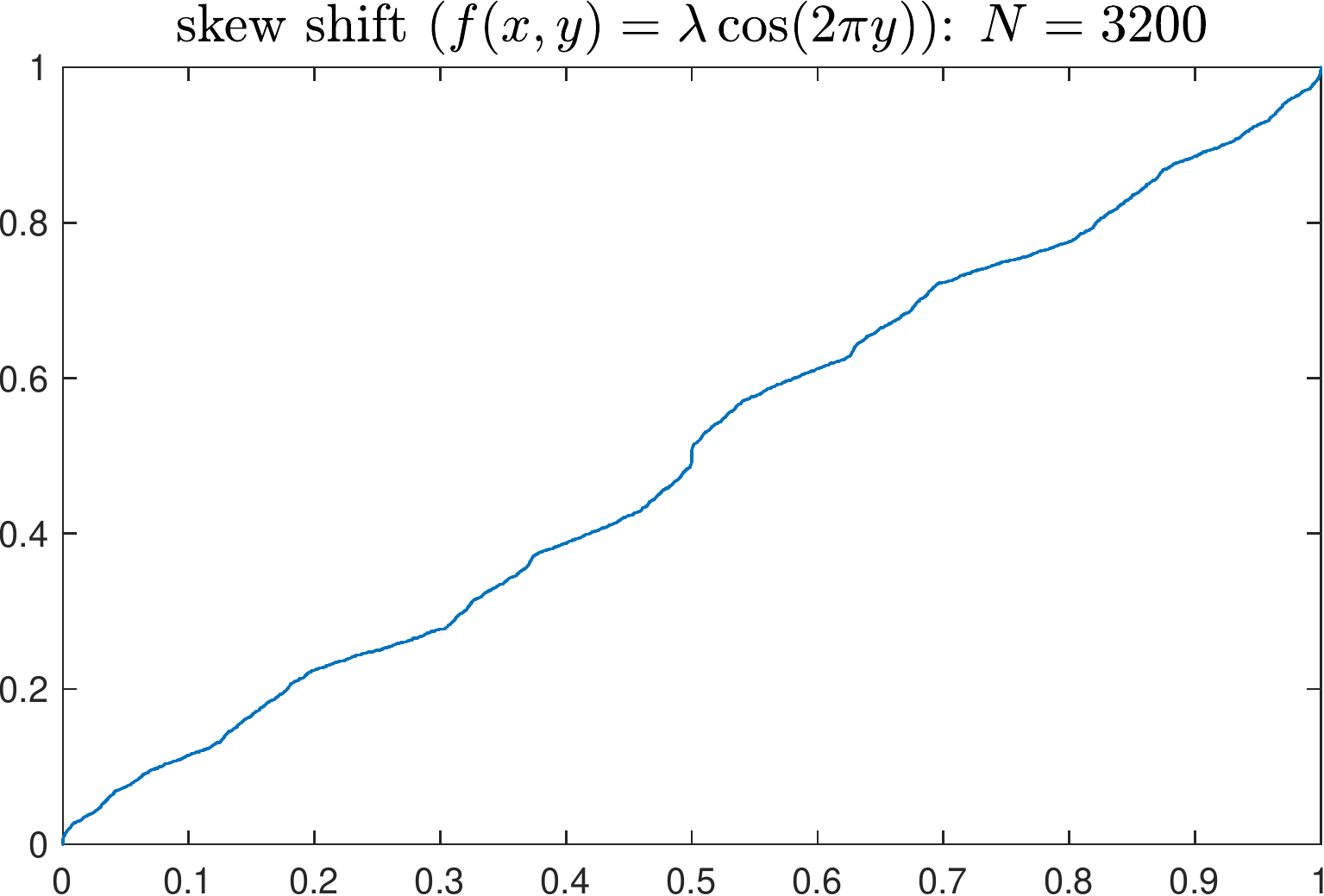}
\end{center}

\begin{center}
\includegraphics[scale=.45]{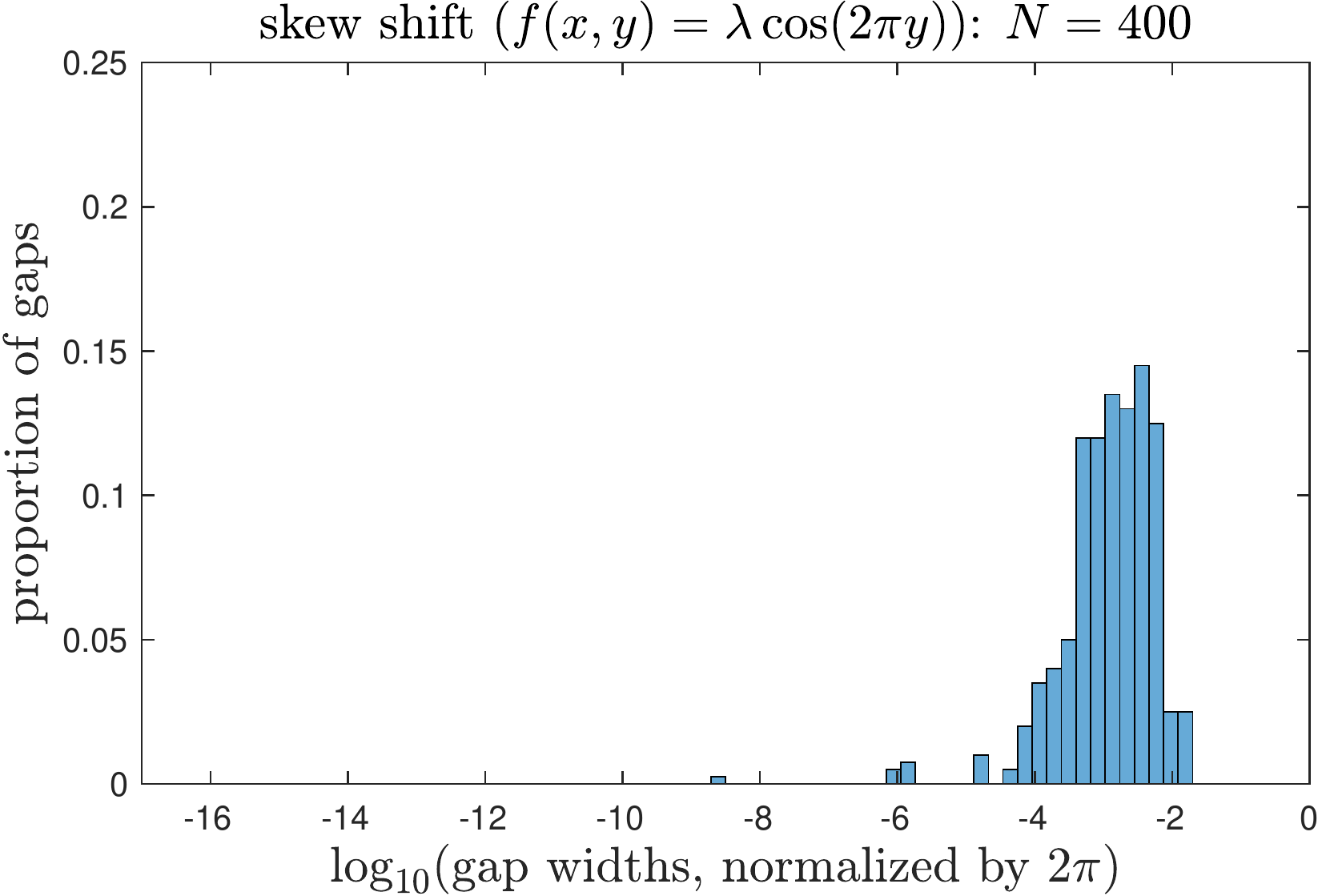}
\quad
\includegraphics[scale=.45]{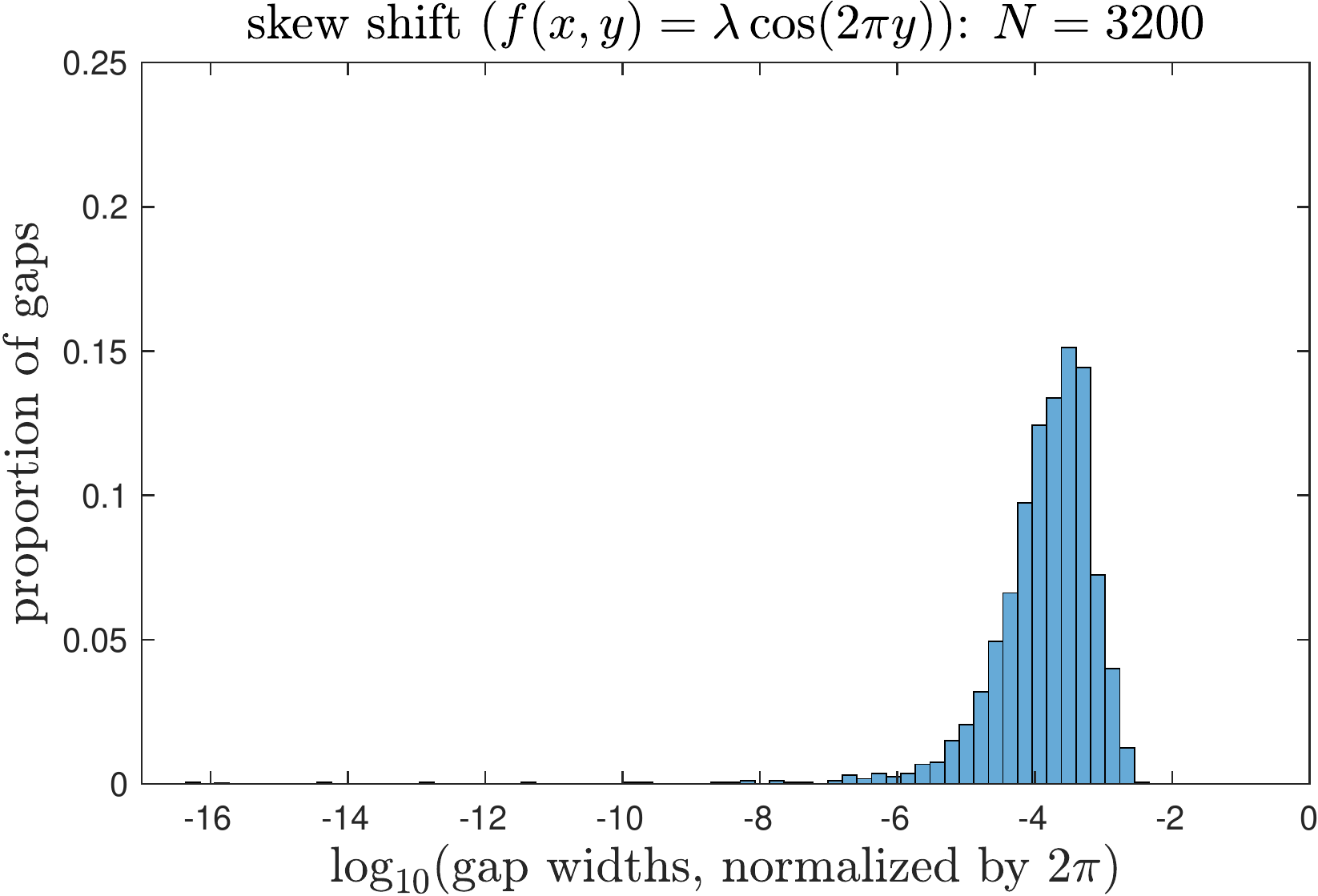}
\end{center}

\end{example}

\begin{example}[Unitary Almost-Mathieu Operator]
We conclude with the unitary almost-Matheiu operator, which was investigated in \cite{CFO, FOZ}. The coefficients are given by choosing $\gamma$ irrational, and constants $0<\lambda_1,\lambda_2<1$, and defining\begin{equation}\alpha_{2n}(x) = \sqrt{1-\lambda_2^2}, \quad \alpha_{2n-1}(x) =  \lambda_1\cos(2\pi(n\gamma+x)) \end{equation}for $x \in \bbT$.
The plots below use $\lambda_1 = 9/10$ and $\lambda_2 = \gamma = 1/\sqrt{2}$.

\begin{center}
\includegraphics[scale=.45]{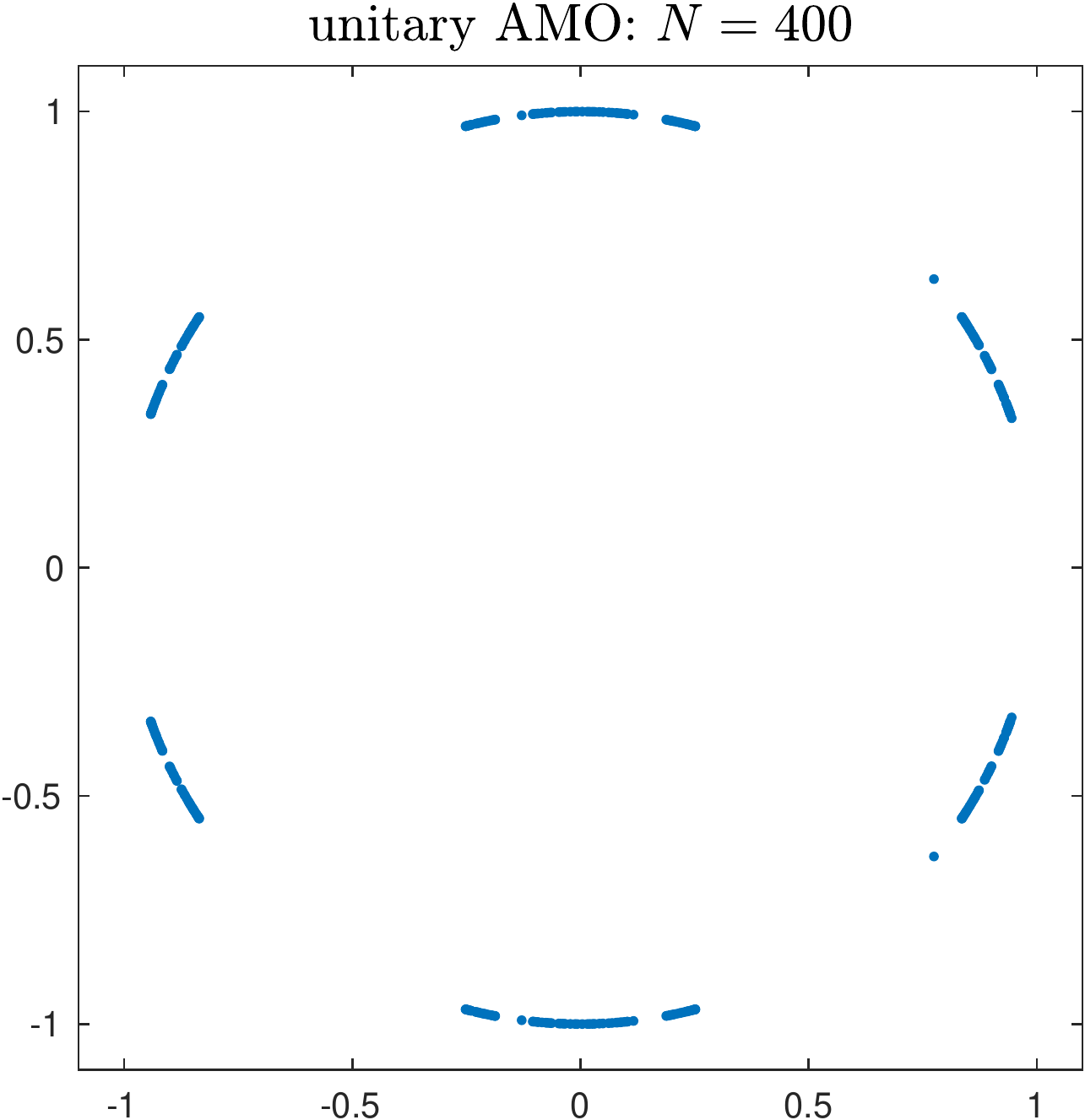}
\quad
\includegraphics[scale=.45]{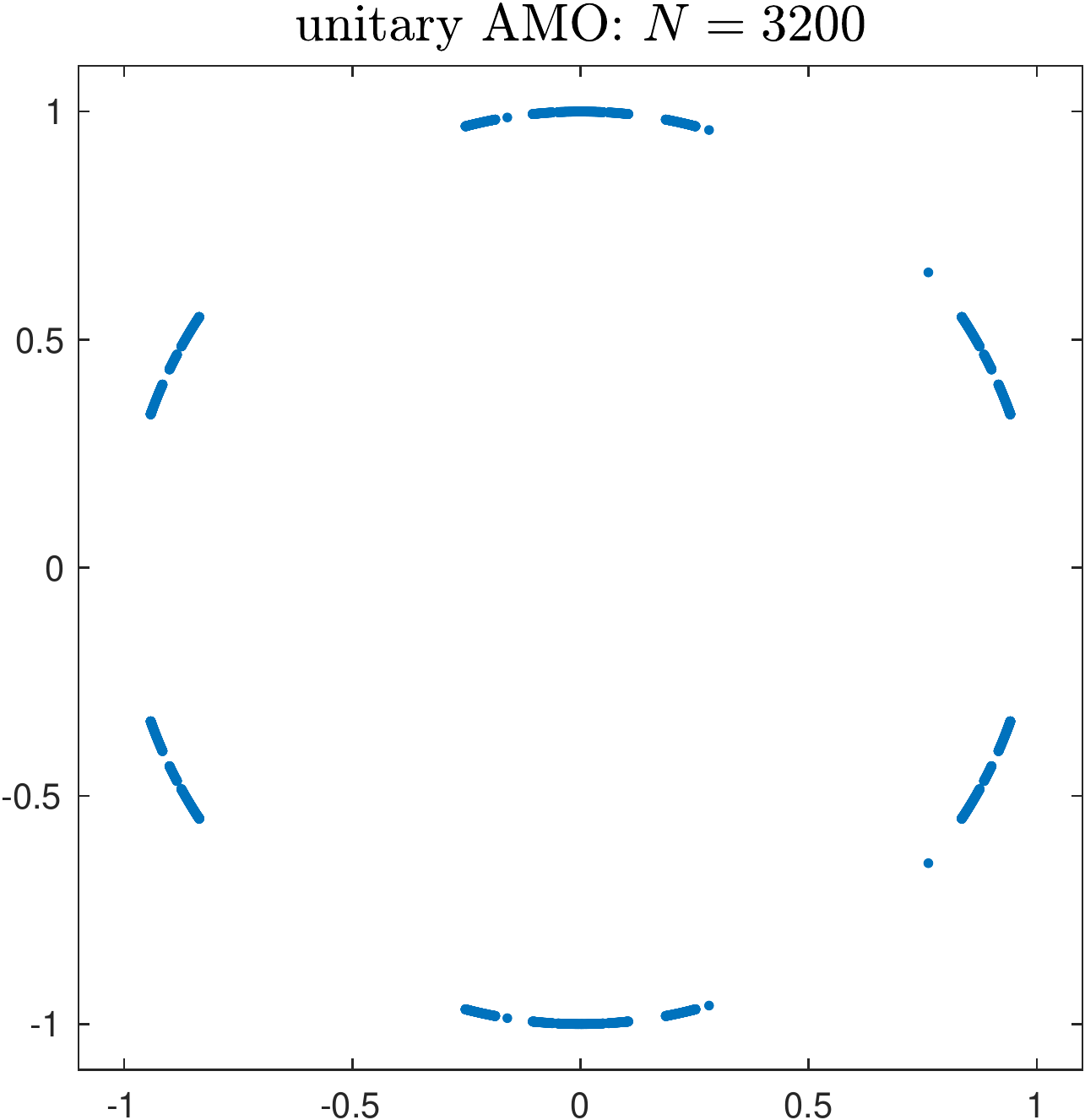}
\end{center}

\begin{center}
\includegraphics[scale=.45]{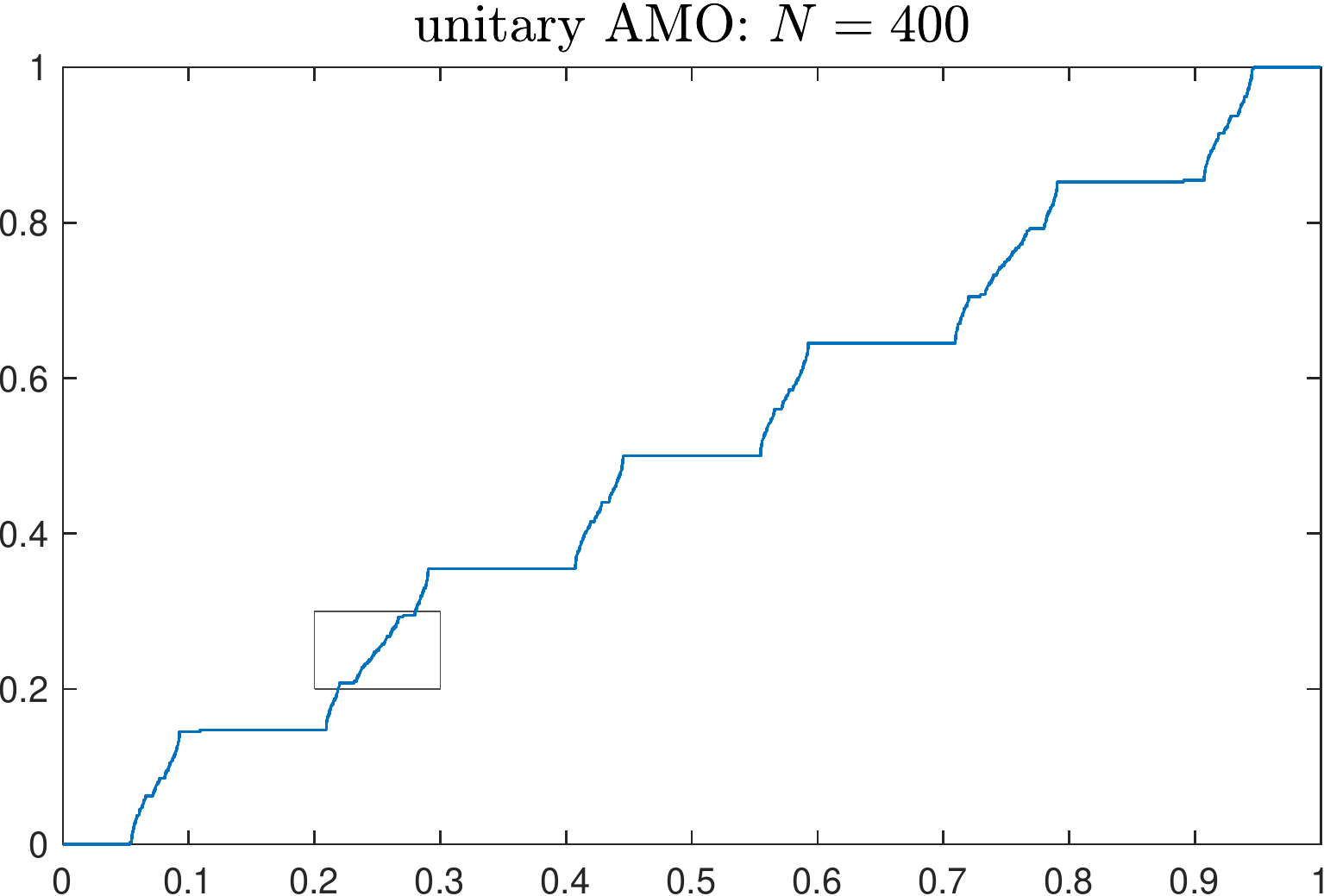}
\begin{picture}(0,0)
  \put(-188,72){\includegraphics[scale=0.18]{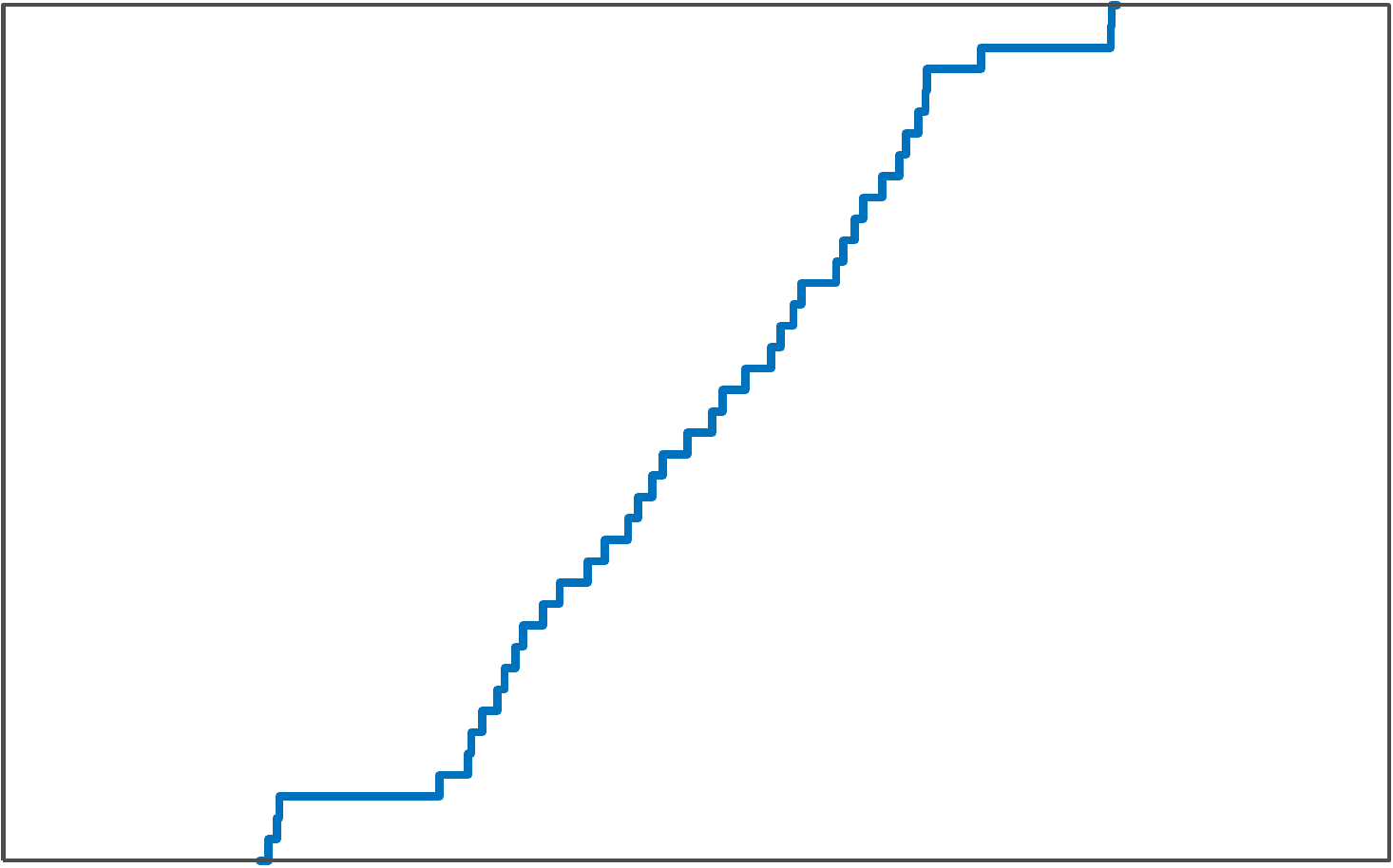}} 
\end{picture}
\qquad
\includegraphics[scale=.45]{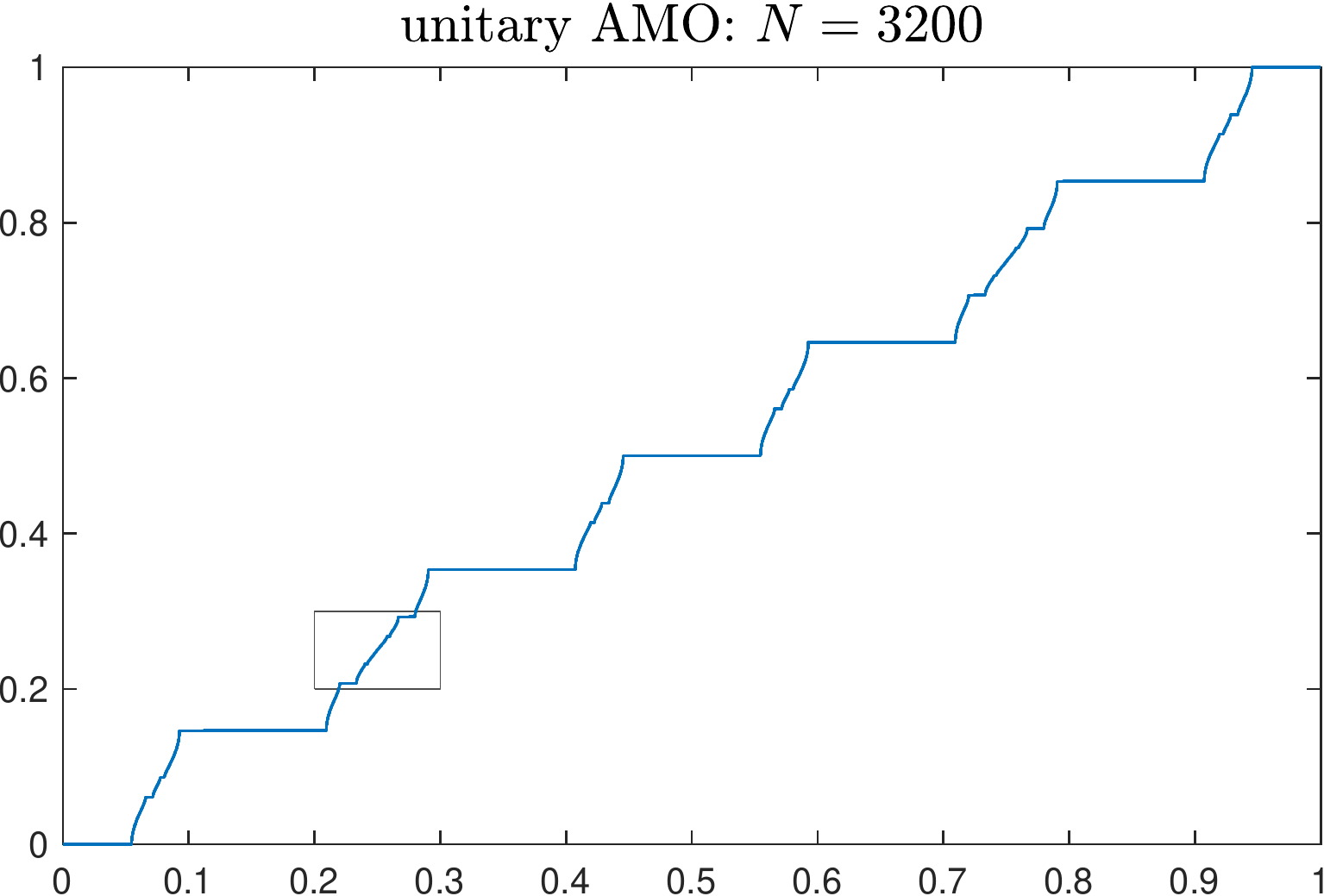}
\begin{picture}(0,0)
  \put(-188,71.75){\includegraphics[scale=0.18]{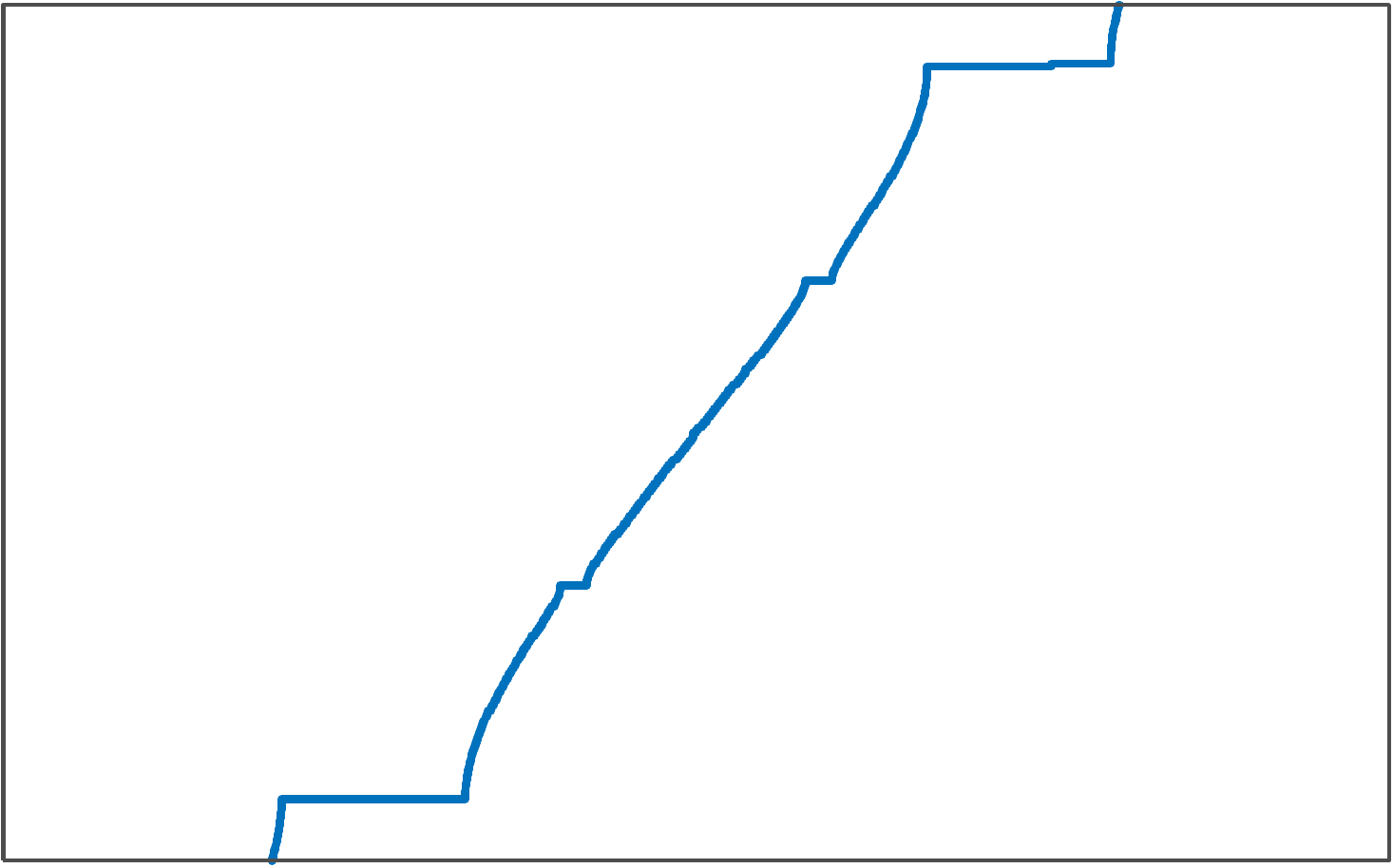}}
\end{picture}
\end{center}

\begin{center}
\includegraphics[scale=.45]{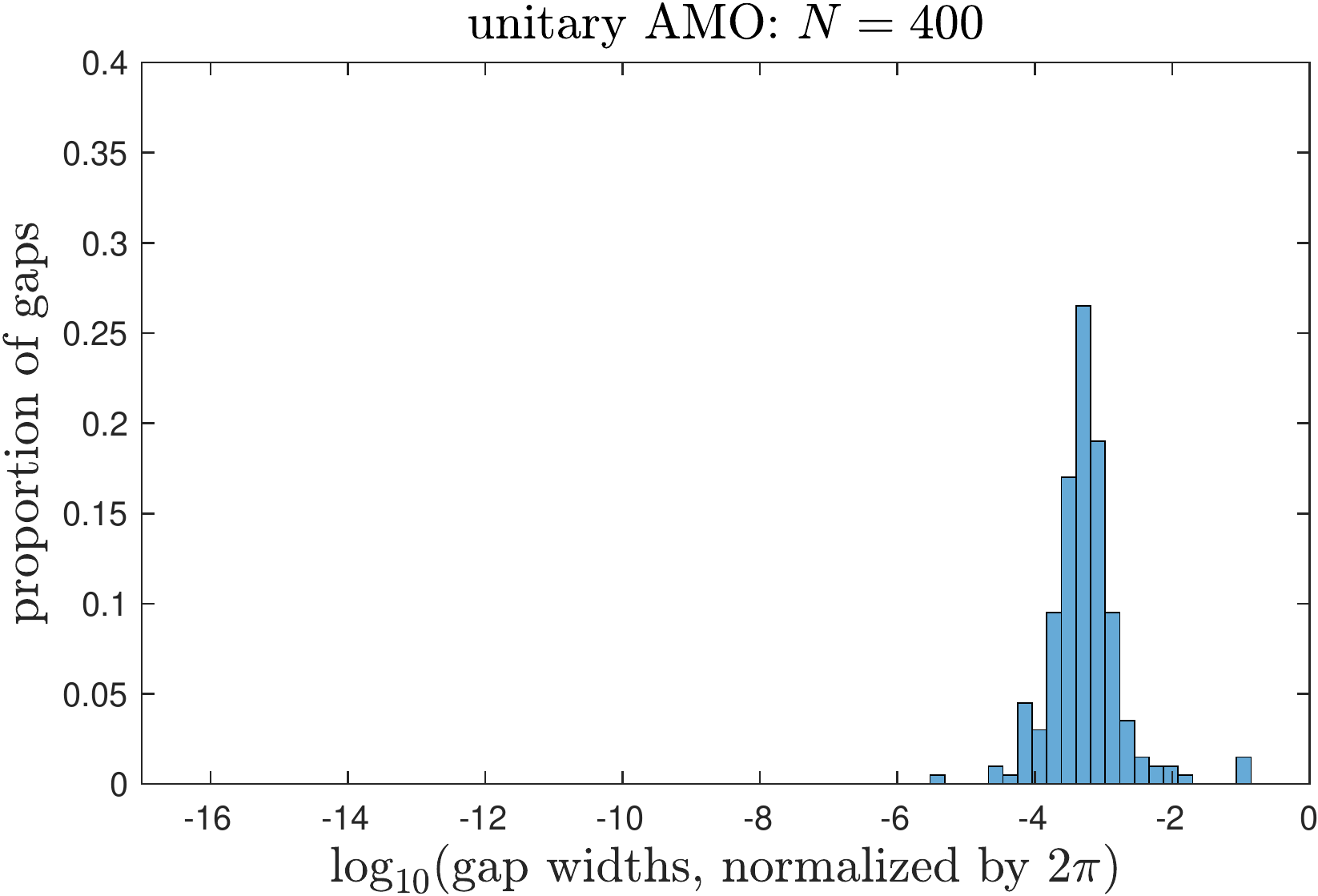}
\quad
\includegraphics[scale=.45]{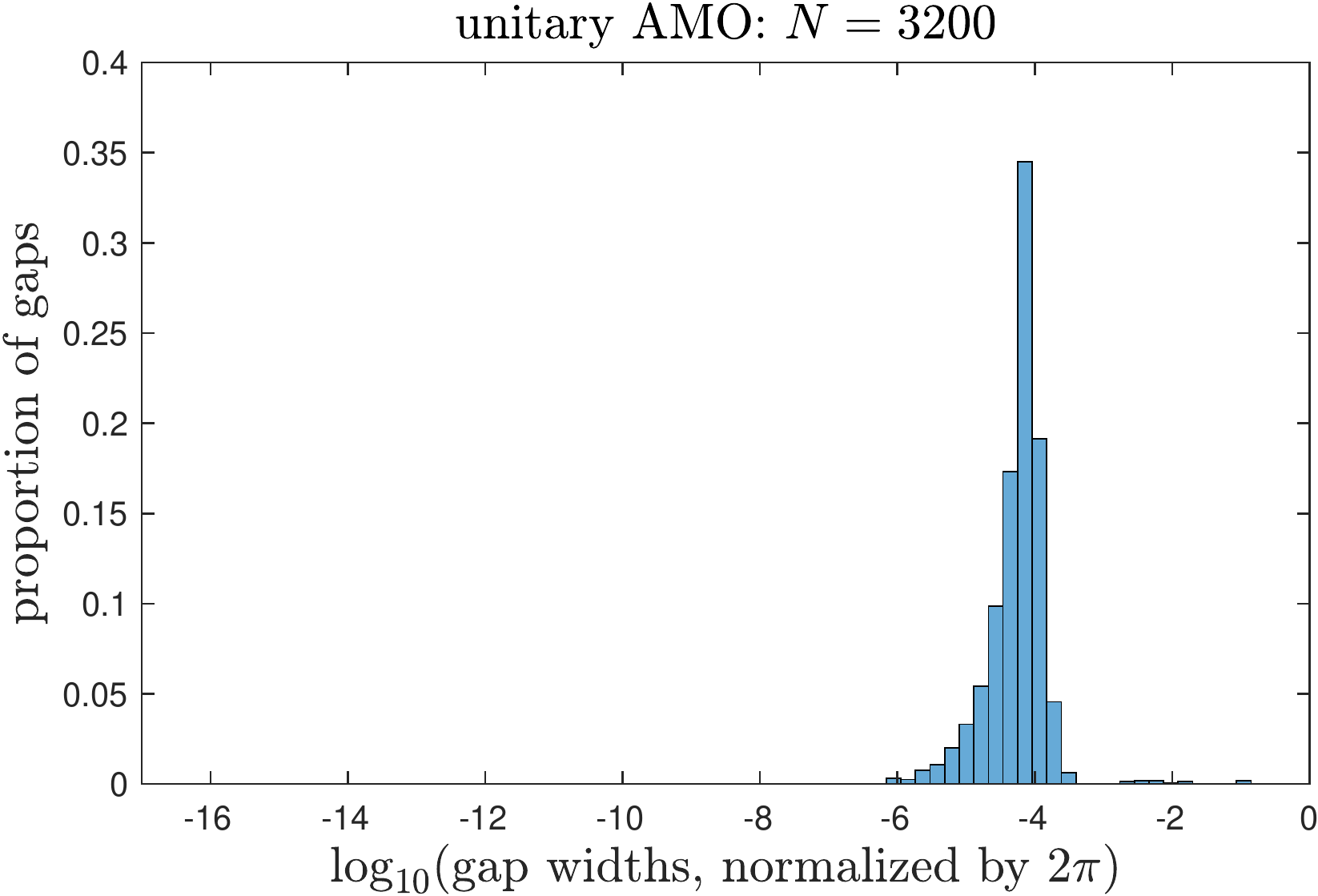}
\end{center}

\end{example}

These plots suggest many interesting problems. We hope this work inspires some readers to study these questions in more detail, with the goal of proving some rigorous results. For instance, it would be very interesting to confirm that the spectra of the skew-shift models proposed above are connected subsets of the circle.

\bibliographystyle{abbrv}

\bibliography{gapbib}

\end{document}